\documentclass[12pt,reqno]{amsart}

\usepackage[T1]{fontenc}
\usepackage{stix2}

\usepackage{mathtools}
\usepackage{mathrsfs}
\usepackage{bbm}

\usepackage[round]{natbib}

\usepackage{amsaddr}

\usepackage[a4paper,left=1.1in,right=1.1in,headsep=2em]{geometry}

\usepackage[colorlinks,citecolor={blue!50!black},linkcolor=blue,urlcolor=blue]{hyperref}

\usepackage{enumitem}
\usepackage{graphicx}
\usepackage{verbatim}
\usepackage{marginnote}

\usepackage[capitalise,nosort]{cleveref}
\crefname{assumption}{Assumption}{Assumptions}
\crefname{figure}{Figure}{Figures}

\usepackage{tikz}
\setlist[enumerate]{itemsep=2pt,topsep=3pt}
\setlist[itemize]{itemsep=2pt,topsep=3pt}
\setlist[enumerate,1]{label={\upshape (\roman*)}}

\numberwithin{equation}{section}
\newtheorem{theorem}{Theorem}[section]
\newtheorem{lemma}{Lemma}[section]
\newtheorem{proposition}{Proposition}[section]
\newtheorem{corollary}{Corollary}[section]
\theoremstyle{definition}
\newtheorem{assumption}{Assumption}[section]

\newtheorem{remark}{Remark}[section]

\newtheorem{example}{Example}[section]

\newcommand{\RR}{\mathbb{R}}
\newcommand{\NN}{\mathbb{N}}
\newcommand{\ZZ}{\mathbb{Z}}

\newcommand{\R}{\mathbb{R}}

\newcommand{\Z}{\mathbb{Z}}
\newcommand{\diff}{\mathop{}\!\mathrm{d}}
\newcommand{\PP}{\mathbb P}
\newcommand{\EE}{\mathbb E}
\newcommand{\TT}{\mathbb T}

\newcommand{\fF}{\mathcal{F}}

\newcommand{\bB}{\mathscr{B}}

\newcommand{\pP}{\mathscr{P}}

\newcommand{\Xsf}{\mathsf{X}}

\newcommand{\del}{\delta}
\newcommand{\ep}{\epsilon}
\newcommand{\lam}{\lambda}

\newcommand{\precsd}{\preceq_{\textrm{sd}}}

\newcommand{\navy}[1]{\textit{\textcolor{blue!40!black}{#1}}}

\begin{document}

\title[Stationary Distributions in Monotone Markov Models]{Stationary Distributions in Monotone Markov Models: Theory and Applications}

\author{Takashi Kamihigashi}
\address{Center for Computational Social Science, Kobe University}
\email{tkamihig@rieb.kobe-u.ac.jp}

\author{John Stachurski}
\address{National Graduate Institute for Policy Studies}
\email{john.stachurski@anu.edu.au}

\begin{abstract}
Many economic models feature monotone Markov dynamics on state spaces that may
be noncompact. Establishing existence, uniqueness, and stability of stationary
distributions in such settings has required a patchwork of sufficient
conditions, each tailored to specific applications. We provide a single
necessary and sufficient condition: a monotone Markov process 
has a globally stable stationary distribution if and only
if it is asymptotically contractive and has a tight trajectory. This
characterization covers both compact and noncompact state spaces, discrete and
continuous time, and extends to nonlinear Markov operators that depend on
aggregate state.  We demonstrate the result through
applications to wage dynamics, Bayesian learning with belief
shocks, and income processes that generate Pareto tails.
\end{abstract}


\maketitle

\setcounter{page}{1}

\section{Introduction}\label{sec:intro}

Many models in economics and finance have monotonicity properties with respect
to state variables.  A general methodology for studying dynamics of such models
in a Markov setting was provided in the frequently cited paper of
\cite{hopenhayn1992stochastic}, which in turn extended earlier results by
\cite{stokey1989recursive}, and \cite{razin}.  The monotone mixing condition
used in that paper has become a standard approach to establishing stability
properties over a large range of applications, including international trade,
human capital, business cycles, inequality, intergenerational mobility, and
labor markets, using models ranging from overlapping generations and competitive
search to mean field games (see, e.g., \cite{chatterjee2010stochastic},
\cite{hidalgo2009endogenous}, \cite{samaniego2008technical},
\cite{marcet2007incomplete}, \cite{antunes2007startup},
\cite{morand2007stationary}, \cite{le2022managing}, \cite{light2022mean},
\cite{balbus2025markov}, and \cite{kam2025inflation}).

The conditions in \cite{hopenhayn1992stochastic} are not universally applicable,
however, since they require a compact state space with least and greatest
elements.  This setup compromises analysis of important features of
economic data, such as heavy tails in cross-sectional distributions.
In addition, many time series models naturally generate unbounded states.
The need to extend \cite{hopenhayn1992stochastic} has motivated
a substantial literature. Subsequent
research has weakened their mixing conditions and replaced compactness of the state
space with ``tightness'' type conditions from the modern literature on Markov
process theory, as well as treating partially monotone models (see, e.g.,
\cite{ks12}, \cite{kamihigashi2014stochastic}, \cite{kamihigashi2016seeking},
\cite{foss2018stochastic}, \cite{kamihigashi2019unified},
\cite{foss2024compressibility}, \cite{light2024monotone}, and \cite{light2026invariant}).

In this paper we bring this literature full circle by providing simple necessary
and sufficient conditions for global stability---that is, existence, uniqueness,
and stability of stationary distributions---in monotone Markov models defined on
either compact or noncompact state space.  These
conditions are built on top of a new fixed point theorem established in the
paper. The fixed point theorem shows that, in any complete preordered metric
space where the metric satisfies a diagonal property, an order-preserving (i.e.,
monotone) operator has a unique fixed point if and only if it is asymptotically
contractive and generates at least one order-bounded trajectory.

In a Markov setting, we show that asymptotic contractivity holds under weak
mixing conditions that can be understood as generalizations of the the
well-known monotone mixing condition in \cite{hopenhayn1992stochastic}. We also
show that, when working with the stochastic dominance partial order, existence
of an order-bounded trajectory is equivalent to existence of a tight
trajectory.  This is a weak version of boundedness in probability, which is a
standard stability-related condition in the Markov process literature (see,
e.g., \cite{meyn2009markov, hahn2024central, ma2025optimal}).  Existence of a tight trajectory is also is also
considerably weaker than assuming, as in \cite{hopenhayn1992stochastic}, that
the state space itself is compact.

In addition, we prove all of these results in an abstract semigroup setting that
\begin{enumerate}
    \item includes both continuous and discrete time, and
    \item allows for nonlinear Markov operators as well as linear ones.
\end{enumerate}
Point (i) means that, in the process of establishing our necessary and
sufficient conditions, we also extend \cite{hopenhayn1992stochastic} and much of
the later surrounding literature to handle continuous time models. Point (ii)
means that our results can be applied to nonlinear Markov distribution dynamics,
where individual Markov updates depend on a distribution and the distribution is
determined by the individual Markov laws (see, e.g., \cite{light2026invariant}).


We demonstrate the scope of our results through several economic applications.
In the compact state space setting, we revisit the monotone mixing condition of
\cite{hopenhayn1992stochastic}, providing sharper results including exponential
convergence rates and ergodicity, and extending the framework to continuous
time. We illustrate with a continuous-time job ladder model of wage dynamics.
For noncompact state spaces, we develop applications to piecewise deterministic
Markov processes (PDMPs), a class of continuous-time models that combine smooth
deterministic dynamics with random jumps. PDMPs arise naturally in models of
wealth accumulation or firm dynamics, and are well-suited to capturing
the role of jumps and discrete shocks in continuous time.  

Our applications include a continuous-time job ladder model of wage dynamics,
extending the framework of \cite{burdett1998wage} and
\cite{moscarini2013stochastic} by allowing state-dependent wage offers; a
model of Bayesian learning subject to belief shocks, as arise in models of
countercyclical uncertainty and forecast bias
\citep{orlik2014understanding, cogley2008drifting}; and two models of income
dynamics---a pure jump process and a PDMP with deterministic drift---both of
which generate Pareto tails in the stationary distribution, complementing
recent work on stationary distributions in continuous-time
heterogeneous-agent models \citep{gabaix2016dynamics} and on Pareto
exponents in multiplicative economies \citep{beare2022determination}.

On a theoretical level, the closest papers to ours are
\cite{kamihigashi2014stochastic} and \cite{foss2024compressibility}, who prove
that global stability holds under a combination of (a) order-theoretic mixing
conditions and (b) side conditions that imply existence of a stationary
distribution. Under matching assumptions on the underlying state space, both of
these results are special cases of our abstract stability result in \cref{t:bk},
which shows that asymptotic contractivity and existence of an order-bounded
trajectory are enough for global stability.  (The results in
\cite{kamihigashi2014stochastic} and \cite{foss2024compressibility} both imply
asymptotic contractivity and order boundedness of at least one trajectory.) At
the same time, \cref{t:bk} goes further, admitting operators acting on arbitrary
preordered spaces, providing necessary conditions for global stability, and
handling continuous time as well as discrete time.  While
\cite{kamihigashi2014stochastic} is essentially superseded by the results in
this paper, \cite{foss2024compressibility} is complementary, adding valuable
sufficient conditions for order boundedness and asymptotic contractivity when
time is discrete.

The remainder of the paper is organized as follows. \cref{sec:prelim}
develops the abstract fixed point theory for order-preserving semigroups.
\cref{ss:md} specializes to probability spaces, introducing the
Bhattacharya metric and connecting tightness to order boundedness.
\cref{s:pmm} covers stochastic kernels and transition probability functions.
\cref{s:mmc} treats the monotone mixing condition on compact state spaces,
with an application to wage dynamics. \cref{ss:ksdis} extends to noncompact
state spaces in discrete time, with an application to Bayesian learning
with belief shocks. \cref{s:nssct} develops the continuous-time
counterpart, with an application to income dynamics. \cref{s:pdmp}
develops the PDMP framework and applies it to income dynamics with
deterministic drift. Proofs are collected in \cref{s:proofs}.

\section{Stability of Asymptotic Contractions}\label{sec:prelim}

We begin with background on preordered metric spaces and the diagonal
property.  We then introduce semigroups and the notions of asymptotic
contractivity and global stability, before stating our main results.

\subsection{Background}

Let $\Xsf$ be a set.  If $A$ is any subset of $\Xsf$, then $A^c$ refers to its
complement in $\Xsf$. We recall that a binary relation $\preceq$
on $\Xsf$ is called a \navy{preorder} if it is transitive and
reflexive. An antisymmetric preorder on $\Xsf$ is called a \navy{partial order}
on $\Xsf$. Given a preorder $\preceq$ on $\Xsf$, each subset $I$ of $\Xsf$
having the form
\begin{equation*}
    I \coloneq [a, b] \coloneq \{x \in \Xsf \, : \, a \preceq x \preceq b\}
    \qquad (a, b \in \Xsf,\; a \preceq b)
\end{equation*}
is called an \navy{order interval} in $\Xsf$.
A set $B \subset \Xsf$ is called \navy{order-bounded} if there exists an order
interval $I$ in $\Xsf$ with $B \subset I$.
A sequence $(x_i) \coloneq (x_i)_{i \in \NN} \subset \Xsf$ is
called \navy{increasing} if $x_i \preceq x_{i+1}$ for all $i$ and
\navy{decreasing} if $x_{i+1} \preceq x_i$ for all $i$.

A map $T$ from $\Xsf$ to itself is called a \navy{self-map} on $\Xsf$. The symbol $T^i$
indicates $i$ compositions of $T$ with itself. A self-map $T \colon \Xsf
\rightarrow \Xsf$ is called \navy{order-preserving} if $x \preceq y$ implies $T
x \preceq T y$.  A \navy{fixed point} of $T$ is an element $x^* \in \Xsf$ such
that $T x^* = x^*$.   

A \navy{preordered metric space} is a triple $(\Xsf, d, \preceq)$ where $(\Xsf,
d)$ is a metric space and $\preceq$ is a preorder on $\Xsf$. For such a space,
the preorder $\preceq$ is called \navy{closed} if $(x_i) \subset \Xsf$, $(y_i)
\subset \Xsf$, $x_i \preceq y_i$ for all $i$, $x_i \to x \in \Xsf$ and $y_i \to
y \in \Xsf$ implies $x \preceq y$.
We say that $d$ has the \navy{diagonal property}
with respect to $\preceq$ if, for all $a, b$ in $\Xsf$,
\begin{equation}\label{eq:diag}
    d(x, y) \leq d(a, b) 
    \quad \text{whenever} \quad
    x, y \in [a, b].
\end{equation}

\begin{example} \label{ex:becarre}
    Let $\Xsf$ be a Banach lattice with norm $\| \cdot \|$, partial order
    $\leq$, and absolute value function $| \cdot |$.
    Fix $a, b$ in $\Xsf$ and $x, y \in [a, b]$.  Since $-x \leq -a$
    and $y \leq b$, we have $y-x \leq b - a$.
    A similar argument gives $x-y \leq b - a$.  Hence $|x-y| \leq b - a = |b-a|$.
    Norms on Banach lattices preserve absolute order, so
    $\|x-y\| \leq \|a - b\|$. 
\end{example}

\begin{example}\label{ex:rfa}
    Let $A$ be any set, let $\Xsf$ be a set of real-valued functions on
    $A$, and let $\preceq$ be the pointwise order on $\Xsf$. If $d$ has the
    form $d(x,y) = e(|x-y|)$ for some non-decreasing function $e$ mapping into
    $\R_{+} \coloneq [0,\infty)$, then $d$ satisfies the diagonal
    property. Indeed, if we fix $x, y, a, b \in \Xsf$ with
    $x, y \in [a, b]$, then $a - b \preceq x-y \preceq b - a$ and hence $|x - y|
    \preceq |b - a|$.  Since $e$ is non-decreasing on $\Xsf$, this yields
        $d(x,y) = e(|x - y|) 
        \leq e(|b - a|) 
        = e(|a - b |) 
        = d(a,b)$.
\end{example}

\cref{ex:becarre,ex:rfa} show that the diagonal property also holds in most of the
classical spaces and any of their subsets (endowed with the same metric and
preorder).  Obvious special cases include
\begin{itemize}
    \item $\Xsf = \RR^n$, where $d$ is Euclidean distance and $\preceq$
        is the usual pointwise partial order,
    \item $\Xsf =$ all bounded, real-valued functions on a given set $M$,
        where $d$ is the supremum distance and $\preceq$ is the pointwise partial order, and
    \item $\Xsf$ is an $L_p$ space, $d$ is the $L_p$ distance, 
        and $\preceq$ is the almost everywhere pointwise partial order (see,
        e.g., \cite{zaanen2012introduction}).
\end{itemize}
\cref{fig:diag} helps illustrate the diagonal property in the first two cases.

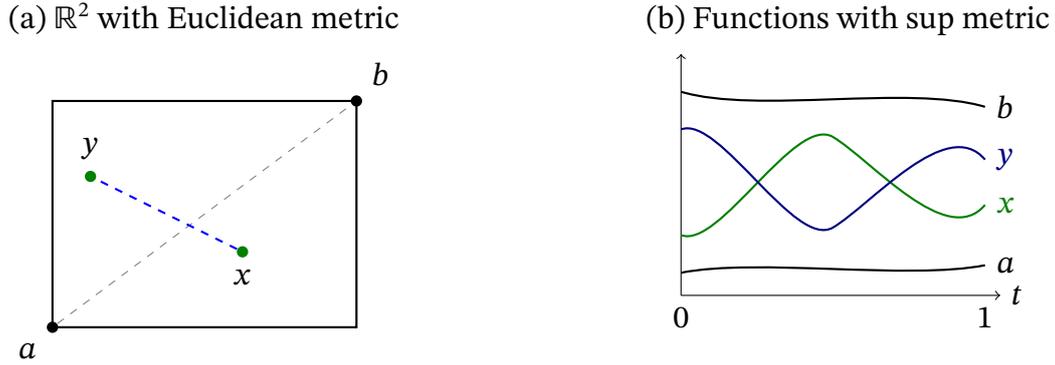
\begin{figure}
    \centering
    \begin{minipage}[t]{0.45\textwidth}
        \centering
        (a) $\R^{2}$ with Euclidean metric

        \vspace{0.5em}
        \begin{tikzpicture}[scale=1.0]
          \draw[thick] (0,0) rectangle (4,3);

          \node[circle, fill=black, inner sep=1.5pt, label=below left:$a$] (a) at (0,0) {};
          \node[circle, fill=black, inner sep=1.5pt, label=above right:$b$] (b) at (4,3) {};

          \draw[dashed, gray] (a) -- (b);

          \node[circle, fill=green!50!black, inner sep=1.5pt, label=below:$x$] (x) at (2.5,1) {};
          \node[circle, fill=green!50!black, inner sep=1.5pt, label=above:$y$] (y) at (0.5,2) {};

          \draw[dashed, blue, thick] (x) -- (y);

        \end{tikzpicture}
    \end{minipage}
    \hfill
    \begin{minipage}[t]{0.45\textwidth}
        \centering
        (b) Functions with sup metric

        \vspace{0.5em}
        \begin{tikzpicture}[scale=1.0]
          \draw[->] (0,0) -- (4.2,0) node[right] {$t$};
          \draw[->] (0,0) -- (0,3.2);
          \node[below] at (4,0) {$1$};
          \node[below] at (0,0) {$0$};

          \draw[thick, black] (0,0.3) .. controls (1,0.5) and (3,0.2) .. (4,0.4) node[right] {$a$};
          \draw[thick, black] (0,2.7) .. controls (1,2.4) and (3,2.8) .. (4,2.5) node[right] {$b$};

          \draw[thick, green!50!black] (0,0.8) .. controls (0.5,0.6) and (1.5,2.4) .. (2,2.1) .. controls (2.5,1.8) and (3.5,0.6) .. (4,1.2) node[right] {$x$};
          \draw[thick, blue!50!black] (0,2.2) .. controls (0.5,2.4) and (1.5,0.6) .. (2,0.9) .. controls (2.5,1.2) and (3.5,2.4) .. (4,1.8) node[right] {$y$};

        \end{tikzpicture}
    \end{minipage}
    \caption{\label{fig:diag} Diagonal property: if $x, y \in [a,b]$, then $d(x,y) \leq d(a,b)$}
\end{figure}

\subsection{Semigroups} Let $(\Xsf, d, \preceq)$ be a preordered metric
space and let $(T_t) \coloneq (T_t)_{t \in \TT}$ be a family of self-maps on $\Xsf$, where
$\TT$ is either $\RR_+$ or $\ZZ_+$.  The case $\TT = \RR_+$ represents
continuous time, while $\TT = \ZZ_+$ represents discrete time. The family
$(T_t)$ is called a \navy{semigroup} on $\Xsf$ when
\begin{equation*}
    T_0 = I
    \quad \text{and} \quad
    T_{s + t} = T_s \circ T_t
    \quad \text{for all } s, t \in \TT.
\end{equation*}
We refer to $\TT$ as the \navy{index set}.
When time is discrete, the semigroup property implies that $T_t$ is just 
$t$ compositions of $T \coloneq T_1$. 

(Note that, unlike some studies of operator
semigroups, no algebraic structure is imposed on the underlying set $\Xsf$.
Moreover, we do not require that $t \mapsto T_t$ is continuous.)

Given $x \in \Xsf$, the map $t \mapsto T_t x$ from $\TT \to \Xsf$
is called the \navy{trajectory} of $x$ under the semigroup $(T_t)$.
The trajectory is called \navy{order-bounded} if its range is contained in an
order interval; that is, if there exists $a, b \in \Xsf$ such that
\begin{equation}\label{eq:bt}
    a \preceq T_t \, x \preceq b \quad \text{for all $t \in \TT$}.
\end{equation}

We say that $x^* \in \Xsf$ is a \navy{stationary point} of $(T_t)$ if $x^*$ is a
fixed point of $T_t$ for all $t \in \TT$. We say that $(T_t)$ is
\begin{itemize}
    \item \navy{asymptotically contractive} on $\Xsf$ if $d(T_t \, x, T_t \, y)
        \to 0$ as $t \to \infty$ for all $x, y \in \Xsf$,
    \item \navy{order-preserving} if $T_t$ is order preserving for all $t \in
        \TT$, and
    \item \navy{globally stable} if $(T_t)$ has a unique stationary point $x^*
        \in \Xsf$
        and
        \begin{equation*}
            d(T_t \, x , x^*) \to 0
            \text{ as } t \to \infty
            \quad \text{for all} \quad
            x \in \Xsf.
        \end{equation*}
\end{itemize}

To simplify terminology, in the case where $\TT = \ZZ_+$ and $T_t$ is the $t$-th
composition of $T \coloneq T_1$ for all $t$, we assign properties
listed above directly to $T$.  For example, we say that $T$ is \navy{globally
stable} if the semigroup $(T_t)_{t \in \ZZ_+} = (T^t)_{t \in \ZZ_+}$ is globally stable.

\subsection{Results}

We can now state our main results in preordered metric space.
Throughout, we take $(\Xsf, d, \preceq)$ to be such a space.
We use the following
assumptions.

\begin{assumption}\label{a:env}
    The metric $d$ is complete and satisfies the diagonal property.
\end{assumption}

\begin{assumption}\label{a:env2}
    The preorder $\preceq$ is closed with respect to $d$.
\end{assumption}

In this setting we have the following result,
which is the main theoretical contribution of the paper.

\begin{theorem}\label{t:bk}
    Let $(T_t)$ be an order-preserving semigroup on $\Xsf$. 
    If \cref{a:env,a:env2} hold, then
    the following statements are equivalent:
    \begin{enumerate}
        \item $(T_t)$ is asymptotically contractive and has at least one
            order-bounded trajectory.
        \item $(T_t)$ is globally stable.
    \end{enumerate}
\end{theorem}

Switching to discrete time, we have the following corollary.
The corollary is obvious given \cref{t:bk}, but nonetheless worth stating due to
its value as a fixed point theorem.

\begin{corollary}\label{c:bk}
    Let $T$ be an order-preserving self-map on $\Xsf$.
    If \cref{a:env,a:env2} hold, then
    the following statements are equivalent:
    \begin{enumerate}
        \item $T$ is asymptotically contractive and has at least one
            order-bounded trajectory.
        \item $T$ is globally stable.
    \end{enumerate}
\end{corollary}

\cref{t:bk} assumes that the preorder $\preceq$ is closed.  This
assumption can be dropped when $(T_t)$ satisfies additional conditions:

\begin{theorem}\label{t:bk2}
    If \cref{a:env} holds and $x \mapsto T_t x$ is
    continuous on $\Xsf$ for all $t \in \TT$, then the equivalence of {\rm (i)} and {\rm (ii)} in \cref{t:bk} is also valid.
\end{theorem}

The proofs of these and other results can be found in the appendix.

\section{Stability on Probability Spaces}\label{ss:md}

Next we consider the special case where the semigroups under investigation act
on a space of probability measures.  In other words, we analyze environments
where distributions are shifted forward in time under given laws of motion.
In this setting, we will be able to relate stability to a compactness-type
condition on trajectories (tightness) that is often applied in the analysis of
distribution dynamics.

\subsection{Background}

Let $S = (S, \leq)$ be a partially ordered
Polish space, that is, $S$ is a separable and completely metrizable topological
space equipped with a closed partial order $\leq$. The Borel sets of $S$ are denoted by $\bB$. A
function $f \colon S \to \RR$ is called \navy{increasing} if $f(x) \leq f(y)$
whenever $x \leq y$. It is called \navy{decreasing} if $f(y) \leq f(x)$ whenever
$x \leq y$.  This terminology extends to sets in $\bB$ by considering their
indicator functions. We let 
\begin{itemize}
    \item $bS$ represent the set of bounded measurable functions from $S$ to $\RR$, 
    \item $cbS$ represent the set of continuous functions in $bS$, 
    \item $ibS$ represent the set of increasing functions in $bS$, 
    \item $\pP(S)$ be the set of all probability measures on $(S, \bB)$.
\end{itemize}
Elements of $\pP(S)$ are also called \navy{distributions}.
For $h \in bS$ and $\phi \in \pP(S)$, we set
\begin{equation*}
    \phi(h) \coloneq \int h d \phi.    
\end{equation*}

We recall that a set $\Phi \subset \pP(S)$ is called \navy{tight} if,
for all
        $\epsilon > 0$, there exists a compact $K \subset S$ with $\phi(K^c) \leq
        \epsilon$ for all $\phi \in \Phi$.

For any pair $\phi, \psi \in \pP(S)$, write 
\begin{equation*}
    \phi \precsd \psi 
    \quad \iff \quad
    \phi(h) \leq \psi(h) \text{ for all $h \in ibS$}.
\end{equation*}
This is the usual notion of (first order) \navy{stochastic dominance}. Under our
assumptions on $(S, \leq)$, the relation $\precsd$ is a partial order on
$\pP(S)$ (see, e.g., \cite{kamae1978stochastic}). 

Throughout the paper, we impose the following restriction on $S$.

\begin{assumption}\label{a:pos}
    A subset of $S$ is compact if and only if it is closed and order bounded.
\end{assumption}

Here is a simple alternative characterization that is helpful for the proofs.

\begin{lemma}\label{l:cobeq}
    Assumption \ref{a:pos} holds if and only if order intervals in $S$ are
    compact and compact subsets of $S$ are order-bounded.
\end{lemma}

\begin{proof}
    Suppose first that \cref{a:pos} holds.  Order intervals are closed (since
    $\leq$ is closed) and order-bounded, hence compact by \cref{a:pos}.  If $K$ is compact, then $K$
    is order-bounded by \cref{a:pos}.  Conversely, suppose that order intervals
    are compact and compact sets are order-bounded.  If $K$ is compact, then $K$
    is closed (since $S$ is Hausdorff) and order-bounded (by assumption).  If
    $K$ is closed and order-bounded, then $K \subseteq [a,b]$ for some $a, b \in
    S$.  Since $[a,b]$ is compact and $K$ is a closed subset, $K$ is compact.
\end{proof}

\begin{example}\label{ex:pos}
    \cref{a:pos} holds when $S$ is a product of intervals in $\RR^n$ with the
    componentwise partial order and Euclidean metric. This includes $\RR^n$,
    $[0,\infty)^n$, $(0,\infty)^n$, $(0,1)^n$, $[0,1]^n$, and mixed cases such as $[0,1)
    \times (0,\infty)$.
\end{example}

\cref{a:pos} has two useful implications for the results below.  One is
completeness of the Bhattacharya metric introduced below. Another is an
equivalence between tight and order bounded subsets of $\pP(S)$:

\begin{proposition}\label{p:tiob}
    $\Lambda \subset \pP(S)$ is tight if and only if $\Lambda$ is order bounded
    in $(\pP(S), \precsd)$.
\end{proposition}

Below, \cref{p:tiob} will be used to connect tightness with global stability.

In this paper, when considering stability of maps over distribution space, we
always work with the \navy{Bhattacharya metric} $\beta$ on $\pP(S)$, which is defined by
\begin{equation}\label{eq:tvsio}
    \beta(\phi, \psi) 
    = \sup_{h \in ibS, \; |h| \leq 1} |\phi(h) - \psi(h)|.
\end{equation}
\cite{kamihigashi2019unified} show that $\beta$ is a complete metric on $\pP(S)$
under \cref{a:pos} (see Theorem~4.1).

\begin{lemma}\label{l:mdi}
    The metric $\beta$ satisfies the diagonal property on $\pP(S)$.
\end{lemma}

\begin{proof}
    Fix $\phi, \psi, \ell$ and $u$ in $\pP(S)$ with $\phi, \psi \in [\ell, u]$.  
    Given $h \in ibS$ with $|h| \leq 1$, we have
    \begin{equation*}
        \phi(h) - \psi(h) \leq u(h) - \ell(h) 
        \quad \text{and} \quad
        \psi(h) - \phi(h) \leq u(h) - \ell(h) .
    \end{equation*}
    It follows that 
        $|\phi(h) - \psi(h)| 
        \leq u(h) - \ell(h) 
        = |u(h) - \ell(h)|
        \leq \beta(\ell, u)$.
    The diagonal property is obtained by taking the sup over $h \in ibS$ with $|h| \leq 1$.
\end{proof}

Another useful observation is as follows.

\begin{lemma}\label{l:betaclosed}
    The partial order $\precsd$ on $\pP(S)$ is closed with respect to $\beta$.
\end{lemma}

\begin{proof}
    Suppose $(\phi_n), (\psi_n) \subset \pP(S)$ with $\phi_n \precsd \psi_n$ for all $n$,
    $\beta(\phi_n, \phi) \to 0$, and $\beta(\psi_n, \psi) \to 0$.
    Fix $h \in ibS$ with $|h| \leq 1$.  From convergence in the Bhattacharya metric, we obtain 
    $\phi_n(h) \to \phi(h)$ and $\psi_n(h) \to \psi(h)$ as
    $n \to \infty$.  We also have $\phi_n(h) \leq \psi_n(h)$ for all $n$.  Hence,
    taking limits, $\phi(h) \leq \psi(h)$.
    This holds for all $h \in ibS$ with $|h| \leq 1$ and, by scaling, extends to all
    $h \in ibS$. Therefore $\phi \precsd \psi$.
\end{proof}

\subsection{Semigroups over Distributions}

Let us now consider a semigroup $(T_t)_{t \in \TT}$ on $\pP(S)$.
The semigroup represents a model that shifts distributions forward in time via
$t \mapsto T_t \phi$, where $\phi$ is some initial distribution.
In all of what follows,
$\pP(S)$ is endowed with the stochastic dominance partial order $\precsd$ and
the Bhattacharya metric $\beta$. Moreover, $S = (S, \leq)$ is itself a partially
ordered Polish space satisfying \cref{a:pos}. In this setting, we can state the
following result.

\begin{theorem}\label{t:act}
    Let $(T_t)$ be an order-preserving semigroup on $\pP(S)$. The following statements are equivalent:
    \begin{enumerate}
        \item $(T_t)$ is globally stable on $\pP(S)$.
        \item $(T_t)$ is asymptotically contractive and has an order-bounded trajectory in $\pP(S)$.
        \item $(T_t)$ is asymptotically contractive and has a tight trajectory in $\pP(S)$.
    \end{enumerate}
\end{theorem}

The statement that $(T_t)$ has a tight trajectory means that there exists at least
one $\phi \in \pP(S)$ such that $(T_t \phi)_{t \in \TT}$ is tight in $\pP(S)$.

Specializing to discrete time, we have the following corollary.

\begin{corollary}\label{c:act}
    If $T \colon \pP(S) \to \pP(S)$ is order-preserving, then the following statements are equivalent:
    \begin{enumerate}
        \item $T$ is globally stable on $\pP(S)$.
        \item $T$ is asymptotically contractive and has an order-bounded trajectory in $\pP(S)$.
        \item $T$ is asymptotically contractive and has a tight trajectory in $\pP(S)$.
    \end{enumerate}
\end{corollary}

In summary, \cref{t:act} and \cref{c:act} provide necessary and sufficient
conditions for global stability of order-preserving semigroups and maps on
distribution space: asymptotic contractivity combined with tightness of at
least one trajectory.

Often the operators that we consider will be Markov operators generated from
transition probability functions. Such operators have other useful properties,
such as linearity with respect to distributions.  We discuss this case in
\cref{s:pmm}. We also note, however, that \cref{t:act} and \cref{c:act} can be
used to study nonlinear maps over distributions, which occur frequently in
economic dynamics (see, e.g., \cite{light2026invariant}).

\section{Probabilistic Markov Models}\label{s:pmm}

In this section we provide background on probabilistic aspects of Markov
dynamics, and on the connection between underlying Markov models and the
operators and semigroups they generate.

\subsection{Stochastic Kernels}

A \navy{stochastic kernel on $S$} is a map $P$ from $S \times \bB$ to $[0,1]$
such that $B \mapsto P(x,B)$ is in $\pP(S)$ for each $x \in S$ and $x \mapsto
P(x,B)$ is measurable for each $B \in \bB$.  The kernel $P$ is called
\navy{increasing} if $x \leq x'$ implies $P(x, \cdot) \precsd P(x', \cdot)$.
This is equivalent to both of the following statements:
\begin{itemize}
    \item $Ph \in ibS$ whenever $h \in ibS$.
    \item $\phi P \precsd \psi P$ whenever $\phi \precsd \psi$.
\end{itemize}
Each stochastic kernel $P$ on $S$ generates two operators.  The first is the
\navy{left Markov operator}, defined via
\begin{equation}\label{eq:lmo}
    \phi \mapsto \phi P, \qquad
    (\phi P) (B)  \coloneq \int P(x,B) \phi(\diff x) \qquad (B \in \bB).
\end{equation}
The second is the \navy{right Markov operator}, defined by
\begin{equation}\label{eq:rmo}
    h \mapsto Ph, \qquad
    (Ph)(x) \coloneq \int h(y) P(x,\diff y)
    \qquad (h \in bS, \; x \in S).
\end{equation}
For convenience, and in line with much of the literature (see, e.g.,
\cite{meyn2009markov}), the same symbol $P$ is used for both operators, as well
as for the stochastic kernel.  In addition, we will drop the terms ``left'' and
``right'', assuming that the meaning will typically be clear from context, and
adding clarification when necessary.

In \eqref{eq:lmo}, the map $\phi \mapsto \phi P$ is understood as updating
distribution $\phi$ to the next period. (One easily confirms that $B \mapsto
(\phi P)(B)$ is a probability measure on $\bB$, so the mapping is from $\pP(S)$
to $\pP(S)$.) In \eqref{eq:rmo}, the value $(Ph)(x)$ is interpreted as the
conditional expectation of $h(X)$ when $X$ is the next period state and $x$ is
the current state.

For the two operators in \eqref{eq:lmo} and \eqref{eq:rmo},
we have the well-known ``duality'' relation
\begin{equation}\label{eq:dual}
    \text{$(\phi P)(h) = \phi(Ph)$ for any $\phi \in \pP(S)$ and any $h \in bS$}.
\end{equation}
(This follows directly from the definitions and Fubini's theorem.) We note that the kernel $P$ is increasing if and only
if the Markov operator $\phi \mapsto \phi P$ is order-preserving on $(\pP(S),
\precsd)$.

The following result parallels the well-known fact that total variation distance
between distributions is not increased after passing them through a Markov
kernel.  It will be useful for our analysis of continuous time models, and also
suggests that asymptotic contractivity should be attainable under some
additional conditions, which turn out to be mixing conditions.

\begin{lemma}\label{lem:nonexp}
    If $P$ is an increasing kernel, then the Markov operator $\phi \mapsto \phi
    P$ is nonexpansive in the Bhattacharya metric:
    \begin{equation}
        \beta(\phi P, \psi P) \leq \beta(\phi, \psi)
        \quad \text{for all} \quad 
        \phi, \psi \in \pP(S).
    \end{equation}
\end{lemma}

\begin{proof}
    Observe that, for probability measures $\phi, \psi \in \pP(S)$, we have
    \begin{equation*}
        \beta(\phi P, \psi P) 
        = \sup_h \left| (\phi P)(h) - (\psi P)(h) \right|
        = \sup_h \left| (\phi )(Ph) - (\psi )(Ph) \right|,
    \end{equation*}
    where the supremum is over $h \in ibS$ with $|h| \leq 1$ and the second equality is by the
    duality relationship in \eqref{eq:dual}. Since $P$ is increasing and $h$ is
    increasing, $Ph$ is increasing. Since $|h|
    \leq 1$, we have $|Ph| \leq 1$, so the last expression is dominated by 
    $\beta(\phi, \psi)$.
\end{proof}

\subsection{Transition Probability Functions}

Following standard terminology, a \navy{transition probability function} is a
family of stochastic kernels $(P_t)_{t \in \TT}$ such that 
\begin{enumerate}
    \item $P_0(x, \cdot) = \delta_x$ for all $x \in S$, and
    \item the \navy{Chapman--Kolmogorov equation} holds:
        \begin{equation}\label{eq:ck}
            P_{s + t}(x, B) = \int P_t(y, B) P_s(x, \diff y)
            \qquad (s, t \in \TT, \; x \in S, \; B \in \bB).
        \end{equation}
\end{enumerate}
When $(P_t)_{t \in \TT}$ is a transition probability function, the family of maps 
$\phi \mapsto \phi P_t$,
\begin{equation}\label{eq:pmp}
    (\phi P_t) (B)  \coloneq \int P_t(x,B) \phi(\diff x) 
    \qquad (B \in \bB, \; t \in \TT)
\end{equation}
forms a semigroup on $\pP(S)$, with the semigroup property
$P_{s + t} = P_s \circ P_t$ following from the Chapman--Kolmogorov relations.
In this setting, we say that $(P_t)_{t \in \TT}$ is a \navy{Markov semigroup on
$S$}. We use the same symbol $(P_t)_{t \in \TT}$ to represent the Markov semigroup and the transition probability
function, since both represent the same object.\footnote{The transition
probability function defines the Markov semigroup on $S$ via \eqref{eq:pmp}.
The semigroup defines the transition probability function via
$P_t(x, B) \coloneq (\delta_x P_t)(B)$ for all $x \in S$, $t \in \TT$ and $B \in \bB$.}

\begin{remark}\label{r:sgn}
    Each transition probability function $(P_t)_{t \in \TT}$ also defines a
    family of operators $h \mapsto P_t h$, indexed by $t \in \TT$, each $P_t$
    mapping $bS$ to $bS$, through the
    right Markov operator relation in \eqref{eq:rmo}.  Some authors have this in
    mind when they refer to the semigroup generated by the transition
    probability function, rather than the family of maps $\phi \mapsto \phi
    P_t$ sending $\pP(S)$ to $\pP(S)$.  This is just a matter of convention.
    Our terminology is similar to \cite{LasotaMackey1994} and \cite{rudnicki2017piecewise}.
\end{remark}

\begin{example}
    In the discrete time case, where we work with a fixed stochastic kernel $P$
    on $S$, we set $P_0 \coloneq I$ and $P_t \coloneq P^t$ for all $t \in \NN$.  The family
    $(P_t)_{t \in \ZZ_+}$ is a transition probability function.
\end{example}

\begin{example}
    Let $S=\RR$ and $P(t, x, \diff y) = p(t, x, y) \diff y$, where
    \begin{equation*}
        p(t,x, y) \coloneq 
        \sqrt{\frac{\theta}{\pi \sigma^2 (1 - e^{-2\theta t})}} 
        \exp\left(
                -\frac{\theta(y - x e^{-\theta t})^2}{\sigma^2(1 - e^{-2\theta t})}
            \right).
    \end{equation*}
    This is the transition density function for the Ornstein--Uhlenbeck process
    \begin{equation}\label{eq:ou-sde}
        dX_t = -\theta X_t \, dt + \sigma \, dW_t, \quad X_0 = x,
    \end{equation}
    where $\theta > 0$ is the mean-reversion rate, $\sigma > 0$ is the volatility
    parameter, and $(W_t)_{t \geq 0}$ is Brownian motion.  The family of maps
    $(P_t)_{t \geq 0}$ defined by 
    \begin{equation*}
        \phi \mapsto \phi P_t
        \quad \text{with} 
        \quad
        (\phi P_t) (B)  \coloneq \int \int_B p(t, x, y) \, \diff y \, \phi(\diff x)    
    \end{equation*}
    is the Markov semigroup on $\RR$ generated by the Ornstein--Uhlenbeck transition
    probability function.  For \eqref{eq:ou-sde}, the solution can be expressed
    as
    \begin{equation}\label{eq:ou-sol}
        X_t = e^{-\theta t} x + \sigma_t Z,
        \quad \text{where} \quad
        \sigma_t \coloneq \sigma \sqrt{\frac{1 - e^{-2\theta t}}{2\theta}}
    \end{equation}
    and $Z$ is a standard normal random variable. Hence, for the right Markov
    operators, we have
    \begin{equation}\label{eq:ou-pth}
        (P_t h)(x) = \EE [h(e^{-\theta t} x + \sigma_t Z)]
        \qquad (x \in S, \; h \in bS).
    \end{equation}
\end{example}

Returning to the general case, a transition probability function $(P_t)_{t \in \TT} = (P_t)$ will be called
\navy{increasing} if $P_t$ is an increasing stochastic kernel for all $t \in
\TT$.  For example, \eqref{eq:ou-pth} implies that $P_t h$ is increasing
whenever $h$ is increasing, so the Ornstein--Uhlenbeck transition probability
function is order-preserving. Evidently, $(P_t)_{t \in \TT}$ is an increasing
transition probability function if and only if it is
order-preserving when viewed as a semigroup on $(\pP(S), \precsd)$, as in
\eqref{eq:pmp}.

We say that a transition probability function $(P_t)$ is \navy{bounded in probability} when the
family of distributions $(P_t(x, \cdot))_{t \in \TT}$ is tight for all $x \in S$.
For example, for the Ornstein--Uhlenbeck process, the representation \eqref{eq:ou-sol}
and $\sigma_t \leq \bar \sigma \coloneq \sigma / \sqrt{2\theta}$ for all $t$
give $|X_t| \leq |x| + \bar \sigma |Z|$, so the family $(P_t(x, \cdot))_{t
\geq 0}$ is tight.

We recall that an $S$-valued stochastic process $(X_t) \coloneq (X_t)_{t \in
\TT}$ supported on a probability space $(\Omega, \fF, \PP)$ is called a
\navy{Markov process} if
\begin{equation*}
    \PP\{X_t \in B \,|\, \fF_s\} = \PP\{X_t \in B \,|\, X_s\}
    \;\; \PP\text{-a.s.}
    \; \text{ for all $s, t \in \TT$ with $s \leq t$ and all $B \in \bB$},
\end{equation*}
where $(\fF_t)_{t \in \TT}$ is the natural filtration generated by $(X_t)$.
A \navy{stopping time} of $(\fF_t)$ is a random variable
$\tau \colon \Omega \to \TT \cup \{\infty\}$ such that $\{\tau \leq t\} \in
\fF_t$ for all $t \in \TT$.
We call $(X_t)$ a
\navy{strong Markov process} if, in addition, the Markov property holds at all
stopping times of $(\fF_t)$: for every $(\fF_t)$-stopping time $\tau$ with
$\PP\{\tau < \infty\} = 1$,
\begin{equation*}
    \PP\{X_{\tau + t} \in B \,|\, \fF_\tau\} = P_t(X_\tau, B)
    \quad \PP\text{-a.s.}
    \quad \text{for all $t \in \TT$ and all $B \in \bB$}.
\end{equation*}
Essentially all standard constructions of Markov processes have the strong
Markov property, including all discrete time processes and all c\`adl\`ag paths
under the standard constructions.

Given a transition probability function $(P_t)_{t \in \TT}$, we say that a Markov process
$(X_t)_{t \in \TT}$ is \navy{$(P_t)_{t \in \TT}$-Markov} when 
\begin{equation*}
    \PP\{X_t \in B \,|\, \fF_s\} = P_{t -s}(X_s, B)
    \quad \text{for all nonnegative $s \leq t$ and all $B \in \bB$}.
\end{equation*}
For the case of discrete time, when $P_t$ is just the $t$-composition $P^t$ for
some fixed stochastic kernel $P$, we will simply say that $(X_t)$ is
\navy{$P$-Markov}.  In other words, $(X_t)_{t \in \ZZ_+}$ is $P$-Markov when
\begin{equation*}
    \PP\{X_{t+1} \in B \,|\, \fF_t\} = P(X_t, B)
    \quad \text{for all $t \in \ZZ_+$ and all $B \in \bB$}.
\end{equation*}
Given $x \in S$, we write $\PP_x$ for the conditional probability $\PP\{ \cdot \,|\, X_0 = x\}$.

\section{Monotone Mixing}\label{s:mmc}

Theorem~2 of \cite{hopenhayn1992stochastic} shows that, in discrete time,
the monotone mixing condition (MMC) is sufficient for global stability of
monotone Markov chains. In this section we extend their result by
(a) establishing an explicit exponential rate of convergence to the
stationary distribution, (b) proving ergodicity for bounded increasing
observables, and (c) generalizing from discrete to continuous time.

We state the main result and its proof in \cref{ss:mmctheory}, and then
apply it to a continuous-time model of wage dynamics in \cref{ss:awd}.

\subsection{Theory}\label{ss:mmctheory}

As before $(S, \leq)$ is a partially ordered Polish space and $\precsd$ is
stochastic dominance. Following \cite{hopenhayn1992stochastic}, we suppose that
$S$ is compact and, in addition, has a greatest element $b$ and a least element
$a$. Note that, under these conditions, \cref{a:pos} is satisfied. Given $x \in
S$, we write $U_x$ for all $y \in S$ with $x \leq y$ and $D_x$ for all $y \in S$
with $y \leq x$.

In line with \cite{hopenhayn1992stochastic}, we say that a
stochastic kernel $P$ on $S$ satisfies the \navy{monotone mixing condition}
(MMC) if there exists an $\hat{x} \in S$ and an $\epsilon > 0$ such that
\begin{equation}
    P(a, U_{\hat{x}}) \geq \epsilon
    \quad \text{and} \quad
    P(b, D_{\hat{x}}) \geq \epsilon.
\end{equation}
In other words, a chain starting at $a$ rises above $\hat{x}$ with positive
probability, while a chain starting from $b$ falls below $\hat{x}$ with positive
probability.

We now present the main result of this section, which extends Theorem~2 of
\cite{hopenhayn1992stochastic}. In the statement,  $S$ has the properties listed at the start of
Section~\ref{ss:mmctheory}.

\vspace{0.4em}

\begin{theorem}\label{t:hp}
    Let $(P_t)_{t \in \TT}$ be a transition probability function on $S$.
    If $(P_t)_{t \in
    \TT}$ is increasing and there exists a positive $u \in \TT$ such that $P_u$
    satisfies the MMC, then $(P_t)_{t \in \TT}$ is globally stable on $\pP(S)$
    under $\beta$ and convergence to the unique fixed point $\phi^*$ is
    exponential: 
    \begin{equation}\label{eq:exp_conv}
        \beta(\phi P_t, \phi^*) \leq C e^{-\alpha t}
        \quad \text{for all $\phi \in \pP(S)$ and $t \in \TT$},
    \end{equation}
    where $C \coloneq 2/(1-\epsilon)$ and $\alpha \coloneq \ln(1/(1-\epsilon))/u$.
    Moreover, when $\TT = \ZZ_+$, each $P$-Markov process is monotone ergodic:
    \begin{equation*}
        \PP
        \left\{
             \lim_{n \to \infty}
             \frac{1}{n} \sum_{t=0}^{n-1} h(X_t) = \phi^*(h)
        \right\} = 1
        \quad \text{for all } h \in ibS.
    \end{equation*}
\end{theorem}

We note that \cref{a:pos} holds in the setting of \cref{t:hp}.
The proof can be found in \cref{ss:proofhp}.

\subsection{Application: wage dynamics}\label{ss:awd}

We develop a continuous-time model of wage dynamics based on a job ladder
framework with job destruction, extending the class of models studied by
\cite{burdett1998wage}, \cite{moscarini2013stochastic}, and
\cite{coles2016equilibrium} by allowing wage offers to be state-dependent.

Throughout, $S \coloneq [0, \bar{w}]$ denotes the
state space for wages, with $\bar{w} \in (0,\infty)$. We work with two stochastic
kernels on $S$.  The first, denoted $Q_u$, governs wage draws upon job destruction.
The second, $Q_e$, governs outside offers while employed.
We assume throughout that both $Q_u$ and $Q_e$ are increasing, so higher current
wages predict higher offers for both employed and unemployed workers.

Let $\delta > 0$ be the job destruction rate and $\lambda > 0$ be the offer
arrival rate while employed. Let $N^d_t$ and $N^e_t$ be independent Poisson
processes with rates $\delta$ and $\lambda$ respectively.
The wage process $(W_t)_{t \geq 0}$ evolves as follows:
\begin{itemize}[nosep]
    \item At jump times of $N^d$, we draw $W_t$ from $Q_u(W_{t-}, \cdot)$
    \item At jump times of $N^e$, we draw $W^e_t \sim Q_e(W_{t-}, \cdot)$ and
        then set $W_t = \max(W_{t-}, W^e_t)$.
\end{itemize}
Between jumps, $W_t$ remains constant.  The max after job arrival means that an
employed worker can accept a new offer or retain their current position.

The process $(W_t)_{t \geq 0}$ is a continuous-time pure jump process. Its
dynamics can be characterized through the infinitesimal generator
$\mathcal{A}$, which describes the instantaneous expected rate of change of
$h(W_t)$ for test functions $h \in bS$. (In discrete time, the analogous
object is the one-step conditional expectation operator $Ph - h$.) In this
setting, $\mathcal{A}$ acts on $h \in bS$ via
\begin{equation*}
    (\mathcal{A}h)(w)
    = \delta \int
        \left[ h(w') - h(w) \right] Q_u(w, \diff w')
        + \lambda \int \left[ h(w \vee w') - h(w) \right] Q_e(w, \diff w').
\end{equation*}
Define the discrete-time \navy{max-chain} $(M_t)_{t \in \ZZ_+}$ starting from $w$ by
$M_0 = w$ and
\begin{equation*}
    M_{t+1} = \max\{M_t, W_{t+1}\}
    \quad \text{where} \quad
    (W_t)_{t \in \ZZ_+} \text{ is $Q_e$-Markov with $W_0 = w$}.
\end{equation*}
Let $Q_{m}$ be the kernel associated with the max-chain.
We can now rewrite $\mathcal{A}$ more compactly as 
$\mathcal{A} = \delta (Q_u - I) + \lambda (Q_{m} - I)$.
Some rearranging gives
\begin{equation*}
    \mathcal{A} = (\delta + \lambda)(K - I)
    \quad \text{where} \quad
    K = \frac{\delta}{\delta + \lambda} Q_u + \frac{\lambda}{\delta + \lambda}
    Q_{m}.
\end{equation*}
This is the infinitesimal generator of a jump chain with jump kernel $K$, so the 
Markov semigroup $(P_t)_{t \geq 0}$ on $[0, \bar w]$ admits the representation
\begin{equation}\label{eq:ptfjk}
    (P_t f)(w) 
    = \sum_{n=0}^{\infty} e^{-(\delta + \lambda)t} 
    \frac{[(\delta + \lambda)t]^n}{n!} (K^n f)(w).
\end{equation}
(See, e.g., \cite{siegrist2022}, Section~16.20.)  

To generate monotone mixing, we impose the following restrictions. 

\begin{assumption}\label{a:quqe}
    There exists a $\hat{w} \in [0, \bar{w}]$, an $n \in \NN$ and an $\epsilon >
    0$ such that
    \begin{equation}
        Q_u^n(\bar{w}, D_{\hat{w}}) \geq \epsilon
        \quad \text{and} \quad
        Q_e^n(0, U_{\hat{w}}) \geq \epsilon.
    \end{equation}
\end{assumption}

The assumption requires the existence of an
intermediate wage level $\hat{w}$ such that unemployed agents can reach this
level and employed agents can fall below this level.
Using \cref{a:quqe} and \cref{t:hp}, we can obtain global stability:

\begin{proposition}\label{p:wagestable}
    If \cref{a:quqe} holds, then the wage process is globally stable, with
    unique stationary distribution $\phi^*$.  Moreover,
    \begin{equation}\label{eq:wage_exp}
        \beta(\phi P_t, \phi^*) \leq C e^{-\alpha t}
        \quad \text{for all } \phi \in \pP(S) \text{ and } t \geq 0,
    \end{equation}
    where 
    \begin{equation}\label{eq:kappa_def}
        C \coloneq 2/(1-\kappa), 
        \quad
        \alpha \coloneq \ln(1/(1-\kappa)),
        \quad \text{and} \quad
        \kappa \coloneq \frac{e^{-(\delta + \lambda)} (\delta \wedge \lambda)^n}{n!} \epsilon.
    \end{equation}
\end{proposition}

\begin{proof}
    First observe that $(P_t)$ is
    order-preserving. To see this, note that $K$ is a convex combination of the increasing
    kernels $Q_u$ and $Q_m$ with state-independent weights, hence increasing.
    Compositions of increasing kernels are increasing, so $K^n$ is increasing
    for all $n$. Finally, \eqref{eq:ptfjk} tells us that $P_t$ is a mixture of $(K^n)_{n \geq 0}$ with
    state-independent Poisson weights, hence increasing.

    Next we show that $P_t$ satisfies the MMC for some $t > 0$. Let $\hat{w}$,
    $n$, and $\epsilon$ be as in \cref{a:quqe}. We first claim that $Q_{m}^n(0,
    U_{\hat{w}}) \geq Q_e^n(0, U_{\hat{w}})$. To see this, generate
    $(W^e_k)_{k \in \ZZ_+}$ from $Q_e$ starting at $W^e_0 = 0$, and let $(M_k)_{k
    \in \ZZ_+}$ be the max-chain constructed from $(W^e_k)$, so that $M_k = \max_{j
    \leq k} W^e_j$. By construction, $M_n \geq W^e_n$, so $W^e_n \geq \hat{w}$
    implies $M_n \geq \hat{w}$. Hence $Q_{m}^n(0, U_{\hat{w}}) \geq Q_e^n(0,
    U_{\hat{w}})$.

    Let $n$ and $\epsilon$ be as in \cref{a:quqe}. The semigroup $P_1$ assigns
    positive probability to the event of $n$ jumps by time $1$, with the type of
    each jump (destruction or offer) determined independently. Starting from
    $\bar{w}$, the event of $n$ consecutive destructions has probability
    $e^{-(\delta+\lambda)} \delta^n / n!$, and conditional on this event, the
    wage follows the $Q_u$-chain for $n$ steps, falling below $\hat{w}$ with
    probability at least $Q_u^n(\bar{w}, D_{\hat{w}}) \geq \epsilon$.
    Similarly, starting from $0$, the event of $n$ consecutive offers has
    probability $e^{-(\delta+\lambda)} \lambda^n / n!$, and conditional on this
    event, the wage follows the max-chain for $n$ steps, rising above $\hat{w}$
    with probability at least $Q_{m}^n(0, U_{\hat{w}}) \geq Q_{e}^n(0,
    U_{\hat{w}}) \geq \epsilon$. Combining these bounds,
    \begin{equation*}
        P_1(\bar{w}, D_{\hat{w}})
        \geq \frac{e^{-(\delta + \lambda)} \delta^n}{n!} \epsilon
        \quad \text{and} \quad
        P_1(0, U_{\hat{w}})
        \geq \frac{e^{-(\delta + \lambda)} \lambda^n}{n!} \epsilon.
    \end{equation*}
    Both bounds are strictly positive, so $P_1$ satisfies the MMC with parameter
    $\kappa$ as defined in \eqref{eq:kappa_def}. Global stability and
    the exponential bound \eqref{eq:wage_exp} now follow from \cref{t:hp}.
\end{proof}

\cref{fig:wage_pdmp} illustrates a sample path along with the empirical
stationary distribution for a wage process evolving on $S = [0, 1]$, with
$\delta = 0.1$ and $\lambda = 0.5$. Upon job destruction with current wage $W$, the new wage is $W'
= W \cdot B_u$ where $B_u \sim \mathrm{Beta}(2, 8)$, so the worker retains a
random fraction of their previous wage, skewed towards large losses.  Upon
receiving an outside offer, the offer wage is $W' = 0.5 + 0.5\, W \cdot B_e$
where $B_e \sim \mathrm{Beta}(8, 2)$, so offers have a floor of $0.5$ and a
state-dependent component skewed towards high values (mean $0.8$). In this
setting, both kernels $Q_u$ and $Q_e$ are increasing: if $w \leq w'$, then $w
B_u \leq w' B_u$ and $0.5 + 0.5\, w\, B_e \leq 0.5 + 0.5\, w'\, B_e$ almost
surely.  In addition, \cref{a:quqe} holds with $\hat{w} = 0.5$ and $n = 1$, since
$Q_u(\bar{w}, D_{\hat{w}})
        = \PP\{B_u \leq 0.5\} \approx 0.98$ and 
        $Q_e(0, U_{\hat{w}}) = \PP\{0.5 \geq \hat{w}\} = 1$.
\cref{p:wagestable} then gives global stability with exponential
convergence.  

In the figure, job destructions are marked in red and accepted offers are marked
with green. The empirical stationary distribution is calculated from a simulated
path of length 200,000.  Convergence of the empirical distribution to the true
stationary distribution follows from the ergodicity result in \cref{t:hp}.

\begin{figure}[ht]
    \centering
    \includegraphics[width=\textwidth]{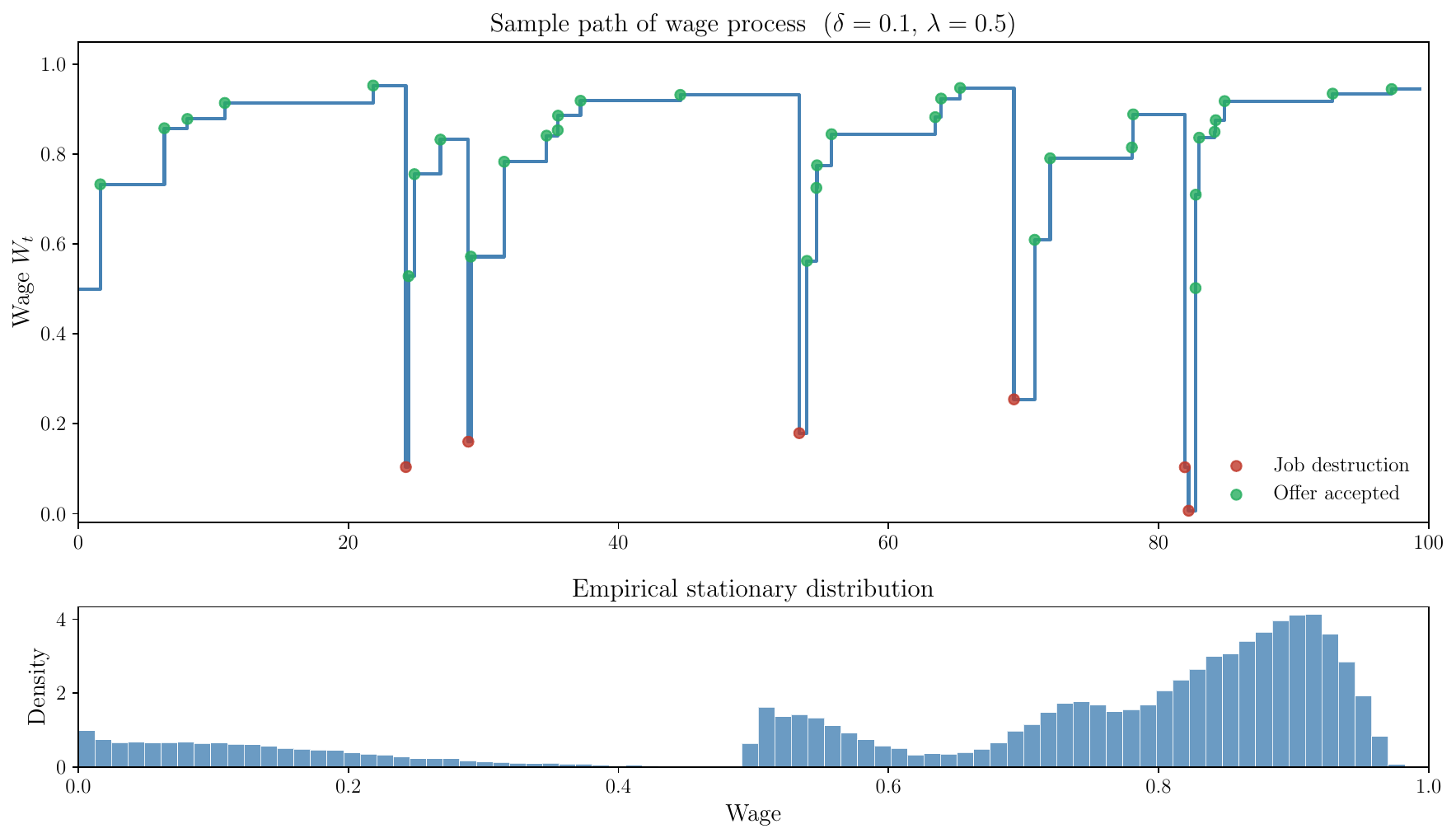}
    \caption{Simulation of the wage dynamics model in \cref{ss:awd}.}
    \label{fig:wage_pdmp}
\end{figure}

\section{Noncompact State Space, Discrete Time}\label{ss:ksdis}

Let's now drop the assumption that the state is compact, 
allowing us to include many more useful processes, and to 
more accurately model properties such as heavy tails in cross-sectional
distributions. Dropping compactness
requires us to adopt at least some restrictions on divergence, so that
probability mass does not vanish towards the ``edges'' of the state space.
Here these restrictions will take the form of tightness conditions,
which are, in turn, connected to order boundedness of trajectories via
\cref{p:tiob}.  We begin in this section by treating discrete time.  Later, in
Section~\ref{s:nssct}, we will cover continuous time in a similar environment.

\subsection{Discrete Time Theory}\label{ss:dtt}

As before, we take $(S, \leq)$ to be a partially ordered Polish space  and $\precsd$ to be stochastic dominance on $\pP(S)$.
The space $(S, \leq)$ is assumed to satisfy \cref{a:pos}. A
distribution $\phi$ is called \navy{excessive} for a stochastic kernel $P$ on
$S$ if $\phi P \precsd \phi$ and \navy{deficient} for $P$ if $\phi \precsd \phi P$.
As usual, $P$ is called \navy{Feller} if $P h \in cbS$ whenever $h \in cbS$.

Let $P$ be a stochastic kernel on $S$. As in \cite{kamihigashi2014stochastic},
we say that $P$ is \navy{order reversing} if there exist two independent
$P$-Markov processes $(X_t)_{t \in \ZZ_+}$ and $(X_t')_{t \in \ZZ_+}$ on a
common probability space such that,
        for every $x, x' \in S$ with $x' \leq x$, there exists a $t \in \ZZ_+$ with
        $\PP_{x,x'}\{X_t \leq X_t' \} > 0$.

Intuitively, order reversing means that no initial ordering is permanent:
regardless of how the two chains start, randomness can eventually reverse their
relative positions. This generalizes the MMC of
\cite{hopenhayn1992stochastic}, which requires mixing at the least and
greatest elements of the state space, to settings where such boundary elements
need not exist.

We can now state our main result for this section.

\begin{theorem}\label{t:orgs}
    Let $P$ be an increasing stochastic kernel on $S$.
    If $P$ is order reversing, then $P$ is globally stable if and only
            if $P$ is bounded in probability.
\end{theorem}

The key idea behind the proof of \cref{t:orgs} is that boundedness in
probability, combined with order reversing, provides a form of monotone mixing
that leads to asymptotic contractivity.  Crucially, the asymptotic contractivity
is with respect to a metric defined via monotone test functions (the
Bhattacharya metric), so a monotone form of mixing suffices---we do not need
anything as strong as irreducibility.  Global stability then follows from
\cref{t:act}, since boundedness in probability implies the existence of a tight
trajectory.

To place \cref{t:orgs} in context, it is helpful to compare it to Theorems~1 and
2 of \cite{kamihigashi2014stochastic}, where it is shown that 
an increasing and order reversing stochastic kernel $P$ on $S$ is globally
stable whenever
\begin{itemize}
    \item[(KS1)] $P$ is bounded in probability and has an excessive or a deficient
        distribution, or
    \item[(KS2)] $P$ is bounded in probability and Feller.
\end{itemize}

\cref{t:orgs} improves both (KS1) and (KS2).  In particular, it
shows that boundedness in probability is enough for global stability when $P$ is
increasing and order reversing. We do not need to separately check for the
existence of an excessive or deficient distribution, as in (KS1), nor do we
require the Feller condition, as in~(KS2).

\cref{t:orgs} implies an ergodicity result:
If $P$ is increasing, order reversing, and is bounded in probability,
then $P$ has a unique stationary distribution $\phi^*$ 
and
\begin{equation*}
    \PP
    \left\{
         \lim_{n \to \infty}
         \frac{1}{n} \sum_{t=0}^{n-1} h(X_t) = \phi^*(h)
    \right\} = 1
\end{equation*}
for all $h \in ibS$, and also for all $h \in cbS$.
This follows from \cref{t:orgs} and Proposition~4.1 in \cite{kamihigashi2016seeking}.

\subsection{Application: Learning with Belief Shocks}\label{ss:bshock}

In macroeconomic and financial modeling, belief shocks are used to generate
persistent macroeconomic effects, such as countercyclical uncertainty,
uncertainty spikes, and forecast biases in inflation data
(see, e.g., \cite{orlik2014understanding},
\cite{cogley2008drifting}, or \cite{suda2010beliefs}). These shocks often
represent exogenous events that degrade accumulated information, or 
replacement of incumbents by new arrivals.  In this section we use
\cref{t:orgs} to study stability and
equilibrium properties in a model of Bayesian learning subject to belief shocks.

In the model, an agent faces uncertainty about
a fixed hidden state $\theta \in \{\ell, h\}$.  Let
$\pi_t \in (0,1)$ represent the agent's belief that $\theta = h$ at time
$t$. In seeking global stability and ergodicity, the end points of the unit
interval are omitted from the state space because both are trivial stationary
points.  We seek to understand the evolution of beliefs on $(0,1)$, which
represents the dynamics of interest for the applications discussed above. Notice
that this precludes use of the methods pioneered by \cite{hopenhayn1992stochastic} and
extended in Section~\ref{s:mmc}, since the state has no least and greatest
element.

The state $\pi_t$ is primarily driven by Bayesian updating while observing
{\sc iid} signals $(Z_t)_{t \in \NN}$ drawn from an unknown distribution. If
$\theta = h$, then the sequence of signals is drawn from density $f_h$.  If
$\theta = \ell$, then the sequence is drawn from density $f_\ell$.
Let $L(z) \coloneq f_h(z)/f_\ell(z)$ denote the likelihood ratio.
Upon observing signal $Z_{t+1}$, beliefs are updated according to Bayes' rule: 
\begin{equation}
    \pi_{t+1} 
     = \frac{\pi_t f_h(Z_{t+1})}{\pi_t f_h(Z_{t+1}) + (1-\pi_t)
     f_\ell(Z_{t+1})}.
\end{equation}

Rather than studying $\pi_t$ directly, we work with 
the log-odds transformation
$\eta_t := \ln(\pi_t/(1-\pi_t))$.
Simple algebra shows that, under this transformation, the dynamics for $\pi_t$ become
\begin{equation}
    \eta_{t+1} = \eta_t + \xi_{t+1},
    \quad \text{where} \quad
    \xi_{t+1} \coloneq \ln L(Z_{t+1}).
\end{equation}
The state space for $(\eta_t)_{t \in \ZZ_+}$ is $\RR$.

We introduce occasional belief shocks that reset the agent's posterior according
to a stochastic kernel $Q$ on $\RR$. Under these shocks,
the belief process $(\eta_t)_{t \in \ZZ_+}$ evolves on $\mathbb{R}$ according to
\begin{equation}\label{eq:bshock}
    \eta_{t+1} 
    = I_{t+1} \cdot R_{t+1} + (1 - I_{t+1}) \cdot (\eta_t + \xi_{t+1}),
\end{equation}
where $(I_t)_{t \in \NN}$ is the belief reset indicator process and
$R_{t+1} \sim Q(\eta_t, \cdot)$ is the reset value.  The stochastic kernel $Q$
is assumed to be increasing. Intuitively, this means that beliefs have some positive
correlation over the reset.  It includes the special case where reset values
are {\sc iid}. The reset indicator process is {\sc iid} and each $I_t$ is drawn
from $\mathrm{Bernoulli}(\rho)$. 

In what follows, we specialize to $\theta = h$, so that each $Z_t$ is drawn
from $f_h$. (Analysis under $\theta = \ell$ is very similar.)
Let $P$ denote the stochastic kernel corresponding to \eqref{eq:bshock}
under this specialization.

\begin{lemma}\label{l:bshock-inc}
    The stochastic kernel $P$ for the belief process is increasing.
\end{lemma}

\begin{proof}
    Take $\eta_1 \leq \eta_2$. Since $Q$ is increasing, $Q(\eta_1, \cdot)
    \precsd Q(\eta_2, \cdot)$.  By Strassen's theorem (see, e.g.,
    \cite{lindvall2002lectures}), there exists a pair of
    random variables $R_1$ and $R_2$ such that the first component $R_1$ is drawn from
    $Q(\eta_1, \cdot)$, the second component $R_2$ is drawn from $Q(\eta_2, \cdot)$, and 
    $R_1 \leq R_2$ almost surely. Let $I$ and $\xi$ be drawn from
    their respective distributions, so that
    \begin{equation*}
         \eta_i' \coloneq I \cdot R_i + (1 - I) \cdot (\eta_i + \xi)
         \; \text{ has distribution } \;
         P(\eta_i, \cdot) \text{ for } i = 1,2.
    \end{equation*}
    If $I = 0$, then $\eta'_1 =
    \eta_1 + \xi \leq \eta_2 + \xi = \eta'_2$. If $I = 1$, then $\eta'_1 = R_1$
    and $\eta'_2 = R_2$. In both cases, $\eta'_1 \leq \eta'_2$.  As a result,
    for any $h \in ibS$,
    \begin{equation*}
        (Ph)(\eta_1) = \EE h(\eta_1') 
        \leq \EE h(\eta_2') = (Ph)(\eta_2).
    \end{equation*}
    This proves that $P$ is increasing.
\end{proof}

We impose the following conditions:
\begin{enumerate}[label=(A\arabic*)]
    \item \label{a:var} The log-likelihood ratio has finite variance:
        $\sigma^2 = \mathrm{Var}(\ln L(Z)) < \infty$.
    \item \label{a:Qbip} For all $\varepsilon > 0$, there exists a compact $K \subset \RR$ such that
        $\inf_{x \in \RR} Q(x, K) \geq 1 - \varepsilon$.
    \item \label{a:unbdd} For every $M > 0$, we have $\PP\{L(Z) > M\} > 0$ when $Z \sim f_h$.
\end{enumerate}
Assumptions \ref{a:var} and \ref{a:unbdd} are naturally satisfied in common
settings. (For example, if $f_h$ and $f_\ell$ are Gaussian densities
with means $\mu_h > \mu_\ell$ and common variance $\sigma^2$, then $\ln L(z) =
(\mu_h - \mu_\ell)(z - \bar \mu)/\sigma^2$ where $\bar \mu = (\mu_h +
\mu_\ell)/2$. This has finite variance and is unbounded above since $z$ has full
support.) Assumption \ref{a:Qbip} is a tightness condition on the reset kernel and holds trivially when resets are {\sc iid}.

\begin{lemma}\label{l:bshock-bip}
    Under {\rm \ref{a:var}} and {\rm \ref{a:Qbip}}, $P$ is bounded in probability.
\end{lemma}

\begin{proof}
    Fix $x \in \RR$ and $\varepsilon > 0$. By \ref{a:Qbip}, choose $K = [a, b]$ such that
    $Q(y, K) \geq 1 - \varepsilon/3$ for all $y$.
    Choose $T$ large enough that $(1-\rho)^T < \varepsilon/3$.
    For $t \geq T$, let $N_t$ denote the number of updates since the
    last reset prior to $t$. The probability that no reset has occurred by time
    $t$ is $(1 - \rho)^t \leq (1-\rho)^T < \varepsilon/3$.
    Conditional on a reset occurring, the post-reset
    state lies in $K$ with probability at least $1 - \varepsilon/3$. Conditional
    on $N_t = n$ and the last reset landing in $K$, we have $\eta_t \in [a +
    S_n, b + S_n]$ where $S_n = \xi_1 + \cdots + \xi_n$. The random variable
    $N_t$ is stochastically dominated by $N \sim \mathrm{Geometric}(\rho)$. By
    Chebyshev's inequality, $\PP\{|S_n| > c\} \leq n\sigma^2/c^2$. Choose $c$
    large enough that
    \begin{equation*}
        \sum_{n=0}^{\infty} \rho(1-\rho)^n \cdot \frac{n\sigma^2}{c^2}
        = \frac{\sigma^2(1-\rho)}{c^2 \rho} < \frac{\varepsilon}{3}.
    \end{equation*}
    Then $\PP_x\{\eta_t \notin [a-c, b+c]\} < \varepsilon$ for all $t \geq T$.
    For $t < T$, each $P^t(x, \cdot)$ is a single distribution and hence tight.
    A finite union of compact sets covers these, completing the proof.
\end{proof}

\begin{lemma}\label{l:bshock-or}
    Under {\rm \ref{a:unbdd}}, the kernel $P$ is order reversing.
\end{lemma}

\begin{proof}
    Let $(\eta_t)$ and $(\eta'_t)$ be independent $P$-Markov processes
    with initial conditions $\eta > \eta'$.
    We will show that $\PP\{\eta_1 < \eta'_1\} > 0$.  To this end,
    let $(I, R, \xi)$ and $(I', R', \xi')$ be the independent shocks for the two processes,
    where $R \sim Q(\eta, \cdot)$ and $R' \sim Q(\eta', \cdot)$.
    Consider the event $A = \{I = 1,\, I' = 0,\, \xi' > R - \eta'\}$.
    On $A$, process~1 resets to $\eta_1 = R$ while process~2 updates to
    $\eta'_1 = \eta' + \xi' > R = \eta_1$, so $\eta_1 < \eta'_1$.
    Since the two processes are independent, $I$, $R$, and $\xi'$ are mutually
    independent.  We have $\PP\{I = 1\} = \rho > 0$, $\PP\{I' = 0\} = 1 - \rho > 0$,
    and, for every realization $R = r$,
    $\PP\{\xi' > r - \eta'\} = \PP\{L(Z) > e^{r-\eta'}\} > 0$ by \ref{a:unbdd}.
    Hence $\PP(A) > 0$, proving that $P$ is order reversing.
\end{proof}

Combining the preceding lemmas with \cref{t:orgs}, we obtain global stability.

\begin{proposition}\label{p:bshock-stable}
    Under {\rm \ref{a:var}--\ref{a:unbdd}}, the belief shock kernel $P$
    is globally stable on $(\pP(\RR), \beta)$.
\end{proposition}

\begin{proof}
    The state space $S = \RR$ satisfies \cref{a:pos} since order intervals
    $[a, b]$ are compact and compact sets are bounded, hence order-bounded.
    By \cref{l:bshock-bip}, $P$ is bounded in probability.
    By \cref{l:bshock-inc}, $P$ is increasing.
    By \cref{l:bshock-or}, $P$ is order reversing.
    Therefore, by \cref{t:orgs}, $P$ is globally stable.
\end{proof}

\cref{fig:belief_shocks} illustrates a simulation of the belief process in
$\pi$-space, with Gaussian signals ($\mu_h = 0.3$, $\mu_\ell = 0$, $\sigma
= 1$), reset probability $\rho = 0.04$, and {\sc iid} Gaussian resets with
mean zero and standard deviation $0.5$ in log-odds space. The time series
shows the characteristic pattern of gradual learning (upward drift, since
the true state is $\theta = h$) interrupted by belief resets that pull the
posterior back toward $\pi = 0.5$. The empirical stationary distribution,
computed from 200{,}000 periods, is right-skewed, reflecting the positive
drift from Bayesian updating under the true state. Convergence of the
empirical distribution to the true stationary distribution follows from the
ergodicity result implied by \cref{t:orgs}.

\begin{figure}[ht]
    \centering
    \includegraphics[width=\textwidth]{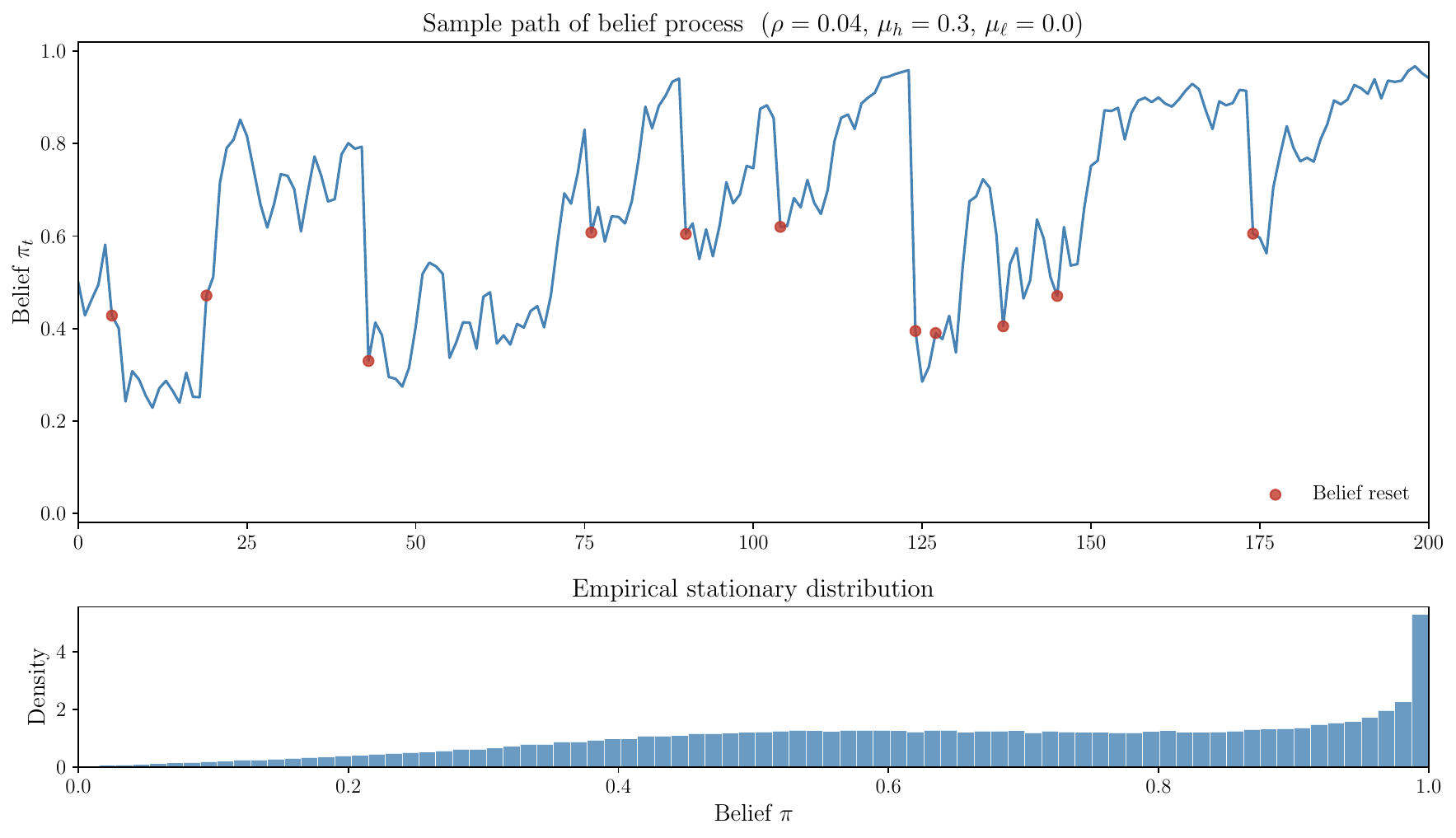}
    \caption{Simulation of the belief shock model in \cref{ss:bshock}.}
    \label{fig:belief_shocks}
\end{figure}

\section{Noncompact State Space, Continuous Time}\label{s:nssct}

Now we state a continuous time result that is closely related to \cref{t:orgs}.
As before, we take $(S, \leq)$ to be a partially ordered Polish space satisfying
\cref{a:pos} and $\precsd$ to be stochastic dominance on $\pP(S)$.
We introduce a continuous-time mixing condition in \cref{ss:ctt} and then
apply it to a model of income dynamics in \cref{ss:purejump}.

\subsection{Continuous Time Theory}\label{ss:ctt}

We start with a mixing condition: a transition probability function 
$(P_t)$ on $S$ will be called \navy{weakly order mixing} if there exists
an $(S \times S)$-valued strong Markov process $((X_t), (X_t'))$ on 
a probability space $(\Omega, \fF, \PP)$ such that
\begin{enumerate}
    \item $(X_t)$ and $(X_t')$ are both $(P_t)$-Markov, and
    \item the stopping time $\tau \coloneq \inf \, \{t \geq 0 : X_t \leq X_t'\}$
        is $\fF$-measurable with
        \begin{equation}\label{eq:omb}
            \PP_{x,x'} \{ \tau < \infty\} = 1
            \quad \text{for all } x, x' \in S.
        \end{equation}
\end{enumerate}

As with order reversing, the intuition is that randomness eventually induces the
order $X_t \leq X_t'$, regardless of initial states $x$ and $x'$. The key
difference between weak order mixing and the order reversing property is that
the two processes need not be independent. This additional flexibility is used
in the applications below, where coupling jump times (processes jump
independently but at the same times) allows us to verify weak order mixing.

The strong Markov and measurability restrictions in the definition of weak order
mixing hold in essentially all applications of interest, such as when the
process has c\`adl\`ag paths under the standard constructions (see, e.g.,
\cite{ethier2009markov}).  Thus, the real restriction in weak
order mixing is \eqref{eq:omb}.

We can now state our main continuous time result for 
a transition probability function $(P_t)$ on $S$.

\begin{theorem}\label{t:orgsnew}
    Let $(P_t)$ be increasing. If
    $(P_t)$ is weakly order mixing, then $(P_t)$ is globally stable if and only
    if $(P_t)$ has at least one tight trajectory.
\end{theorem}

As in the discrete-time case (\cref{t:orgs}), global stability is
characterized by the combination of mixing properties and tightness.
We apply this result to piecewise deterministic Markov processes in
\cref{s:pdmp} below.

\subsection{Application: Income Dynamics Part I}\label{ss:purejump}

In the model of \citet{gabaix2016dynamics}, individual log income follows a
diffusion with constant drift, punctuated by jumps and random resets.  This
model is rather stylized. In practice, most workers' incomes exhibit neither
continuous drift nor positive quadratic variation (e.g., influence from Brownian
motion). Instead, most real-world income processes 
remain constant for nonzero time intervals and then jump due to
promotions, layoffs, retirement, etc.  Here we
study a model of income that mimics these dynamics.
(An additional advantage of the model is ease of simulation: no Euler--Maruyama
style discretization is required.) We analyze stationarity and asymptotics using
\cref{t:orgsnew}.  Then we examine tail properties.

Let $X_t$ represent log income at time $t$.
In our model, \navy{income shocks} arrive at rate $\lam_1 > 0$ and represent raises,
promotions, or better job offers. At such an event, log income jumps from $x$ to
$x + \eta$, where $\eta > 0$ is an {\sc iid} draw from a distribution $\mu$ on
$(0,\infty)$.  Resets arrive at rate $\lam_2 > 0$ and represent job loss, health
shocks, retirement, or career disruptions: the state jumps from $x$ to $h(x) +
\zeta$, where $\zeta$ is an {\sc iid} draw from a \navy{reset distribution} $\nu$ on $\RR$
and $h \in ib\RR$ is a \navy{reset function}.  The assumption that $h$ is increasing
represents the idea that high-wage workers are likely to be better off after
career disruptions than low wage workers.  Pure resets can be obtained by
setting $h$ equal to some constant.

By the superposition property of Poisson processes, the two independent shock streams
can be merged into a single Poisson process with rate $\lam \coloneq \lam_1 + \lam_2$.  At each jump, a reset
occurs with probability $p\coloneq \lam_2/\lam$ and an income shock with probability $q\coloneq \lam_1/\lam$.
Income is constant between jumps, so the process satisfies $X_t = Z_{N_t}$, where
$(N_t)_{t \geq 0}$ is a Poisson process with parameter $\lam$ that counts the number of jumps by time
$t$, and $(Z_n)_{n \in \ZZ_+}$ is the \navy{embedded chain}, defined by 
\begin{equation}\label{eq:purejumpchain}
    Z_{n+1} = 
    \begin{cases}
        Z_n + \eta_{n+1} & \text{with probability } q \; \text{(income shock)},\\
        h(Z_n) + \zeta_{n+1} & \text{with probability } p  \; \text{(reset)}.
    \end{cases}
\end{equation}
Here $(\eta_n)_{n \in \NN}$ and $(\zeta_n)_{n \in \NN}$ are independent and {\sc
iid}, drawn from $\mu$ and $\nu$ respectively. We can now state the following stability result.  

\begin{proposition}\label{p:purejumpstable}
    If the support of $\nu$ is larger than the range of $h$, 
    then the log income process is globally stable: there exists a unique
    stationary distribution $\phi^*$ on $\RR$ and
    \begin{equation}\label{eq:purejumpgs}
        \lim_{t \to \infty} \beta(\phi P_t, \phi^*) = 0
        \quad \text{for all} \quad \phi \in \pP(\RR).
    \end{equation}
\end{proposition}

\cref{p:purejumpstable} obviously implies that income itself is globally stable,
with unique stationary density $\psi^*$ on $(0,\infty)$ given by $\psi^*(y) =
\pi(\ln y)/y$, where $\pi$ is the density of $\phi^*$.

To clarify the assumptions, let $\underline{h} = \inf h$ and $\bar{h} = \sup h$.
The condition on $\nu$ says that there exist $a,b$ in the support of
$\nu$ with $b - a > \bar{h} - \underline{h}$.  If $h$ is constant (pure
resets), this just means that the reset shock is nondegenerate.  If the support
of $\nu$ is all of $\RR$, then the condition is always satisfied, since $h$ is
assumed to be bounded.  (For this case the boundedness assumption on $h$ can
potentially be weakened via a drift condition on $x \mapsto |h(x)|$.
We leave this analysis to future research.)

\begin{proof}[Proof of \cref{p:purejumpstable}]
    Let $(P_t)_{t \geq 0}$ be the transition probability function of the jump
    process $(X_t)_{t \geq 0}$. We
    verify that $(P_t)$ is increasing, weakly order mixing, and has a tight
    trajectory. Global stability then follows from \cref{t:orgsnew}.

    \textit{$(P_t)$ is increasing.}  Fix $t > 0$. On a common probability space,
    we construct two copies of the income process from initial conditions $x
    \leq x'$ using the same jump times, the same shock-type indicators (income
    vs.\ reset), and the same shock values.  For the embedded chains, $Z_0 = x
    \leq x' = Z_0'$. If $Z_n \leq Z_n'$ and the $(n+1)$-th jump is an income
    shock, then
    \begin{equation*}
        Z_{n+1} = Z_n + \eta_{n+1} \leq Z_n' + \eta_{n+1} = Z_{n+1}',
    \end{equation*}
    since $\eta_{n+1} > 0$.  If it is a reset, then
    \begin{equation*}
        Z_{n+1} = h(Z_n) + \zeta_{n+1} \leq h(Z_n') + \zeta_{n+1} = Z_{n+1}',
    \end{equation*}
    since $h$ is increasing.  By induction $Z_n \leq Z_n'$ for all $n$, and
    therefore $X_t = Z_{N_t} \leq Z_{N_t}' = X_t'$.
    It follows directly that $P_t(x, \cdot) \precsd P_t(x', \cdot)$.

    \textit{$(P_t)$ is weakly order mixing.}  To prove this fact, we construct two copies using the
    same jump times and the same shock-type indicators, but independent shock
    values: $(\eta_n, \zeta_n)$ for the first copy and $(\eta_n', \zeta_n')$ for
    the second.  Each marginal is $(P_t)$-Markov.  The joint process $((X_t),
    (X_t'))$ is a pure jump process on $\RR^2$ with constant total rate $\lam$
    and hence strong Markov
    \citep[Theorem~2.4]{rudnicki2017piecewise}.  Since $\{(x, x') \in \RR^2 : x
    \leq x'\}$ is closed, $\tau \coloneq \inf\{t \geq 0 : X_t \leq X_t'\}$ is
    a measurable stopping time.

    Let $A_n$ be the event that the $n$-th jump is a reset and
    $\zeta_n' - \zeta_n \geq \bar{h} - \underline{h}$.
    On $A_n$, we have
    \begin{equation*}
        Z_n
        = h(Z_{n-1}) + \zeta_n
        \leq \bar{h} + \zeta_n
        \leq \underline{h} + \zeta_n'
        \leq h(Z_{n-1}') + \zeta_n'
        = Z_n'.
    \end{equation*}
    The last inequality means that $\tau \leq T_n$ on $A_n$, where $T_n$ is the
    $n$-th jump time. Since the shock type and shock values at step $n$ are
    independent of the filtration $\fF_{n-1}$,
    \begin{equation*}
        \PP\{A_n^c \mid \fF_{n-1}\} \leq 1 - p\del,
    \end{equation*}
    where $\del \coloneq \PP\{\zeta' - \zeta \geq \bar{h} - \underline{h}\}$.
    Hence $\PP\{\tau > T_n\} \leq (1 - p\del)^n$. 
    By our assumption on the support of $\nu$, we have $\del >
    0$ and hence $\PP\{\tau > T_n\} \to 0$ as $n \to \infty$.
    As $T_n \to \infty$ almost surely, it follows that $\PP_{x,x'}\{\tau < \infty\} = 1$.

    \textit{$(P_t)$ has a tight trajectory.}  Fix $x_0 \in \RR$ and $\ep > 0$.
    Since the shock types and values $(\eta_n, \zeta_n)$ are independent of the
    jump times, the embedded chain $(Z_n)$ is independent of the
    counting process $(N_t)$.  Therefore
    \begin{equation*}
        \PP_{x_0}\{X_t \notin D\}
        = \sum_{n=0}^{\infty} \PP\{N_t = n\} \, \PP_{x_0}\{Z_n \notin D\}
        \leq \sup_{n \geq 0} \PP_{x_0}\{Z_n \notin D\},
    \end{equation*}
    so it suffices to find, for each $\ep > 0$, a compact $D$ with
    $\PP_{x_0}\{Z_n \notin D\} < \ep$ for every $n \geq 0$.

    Let $R \sim \mathrm{Geom}(p)$ on $\{0,1,2,\ldots\}$ denote the
    number of income shocks in a generic cycle between consecutive
    resets, and let $W \coloneq \sum_{i=1}^{R} \eta_i$ be the total
    income accumulated in one such cycle.
    Before the first reset (i.e., for $n \leq R$), every jump is an
    income shock, so
    $Z_n = x_0 + \eta_1 + \cdots + \eta_n \leq x_0 + W$.
    At each reset step $k$, the post-reset state satisfies
    $\underline{h} + \zeta_k \leq Z_k \leq \bar{h} + \zeta_k$.
    Between that reset and the next, the chain only accumulates
    positive income shocks, so for any $n$ falling in the same cycle,
    \begin{equation*}
        \underline{h} + \zeta_k \leq Z_n \leq \bar{h} + \zeta_k + W',
    \end{equation*}
    where $W'$ is an independent copy of $W$.
    Since $(\zeta_n)$ and $(\eta_n)$ are {\sc iid} sequences, the
    distributions of $\zeta_k$ and $W'$ do not depend on the cycle.
    It follows that the sequence of distributions $(P^n(x_0, \cdot))_{n \in \NN}$
    of the embedded chain is order bounded in $(\pP(\RR), \precsd)$,
    and hence tight by \cref{p:tiob}. Global stability now follows from \cref{t:orgsnew}.
\end{proof}

\cref{fig:income_purejump} illustrates a simulation of income $Y_t \coloneq \exp(X_t)$
with $h \equiv 0$, $\mu = \mathrm{Exp}(20)$ (mean raise of $5\%$ in
log income), $\lam_1 = 1.0$, $\lam_2 = 0.1$, and {\sc iid} Gaussian reset
shocks $\zeta \sim N(0, 0.09)$.  Income shocks are frequent,
producing a staircase pattern of gradual
accumulation depleted by sharp resets.  The empirical stationary
distribution, computed from 200{,}000 time units, is right-skewed with a heavy
tail. 

\begin{figure}[ht]
    \centering
    \includegraphics[width=\textwidth]{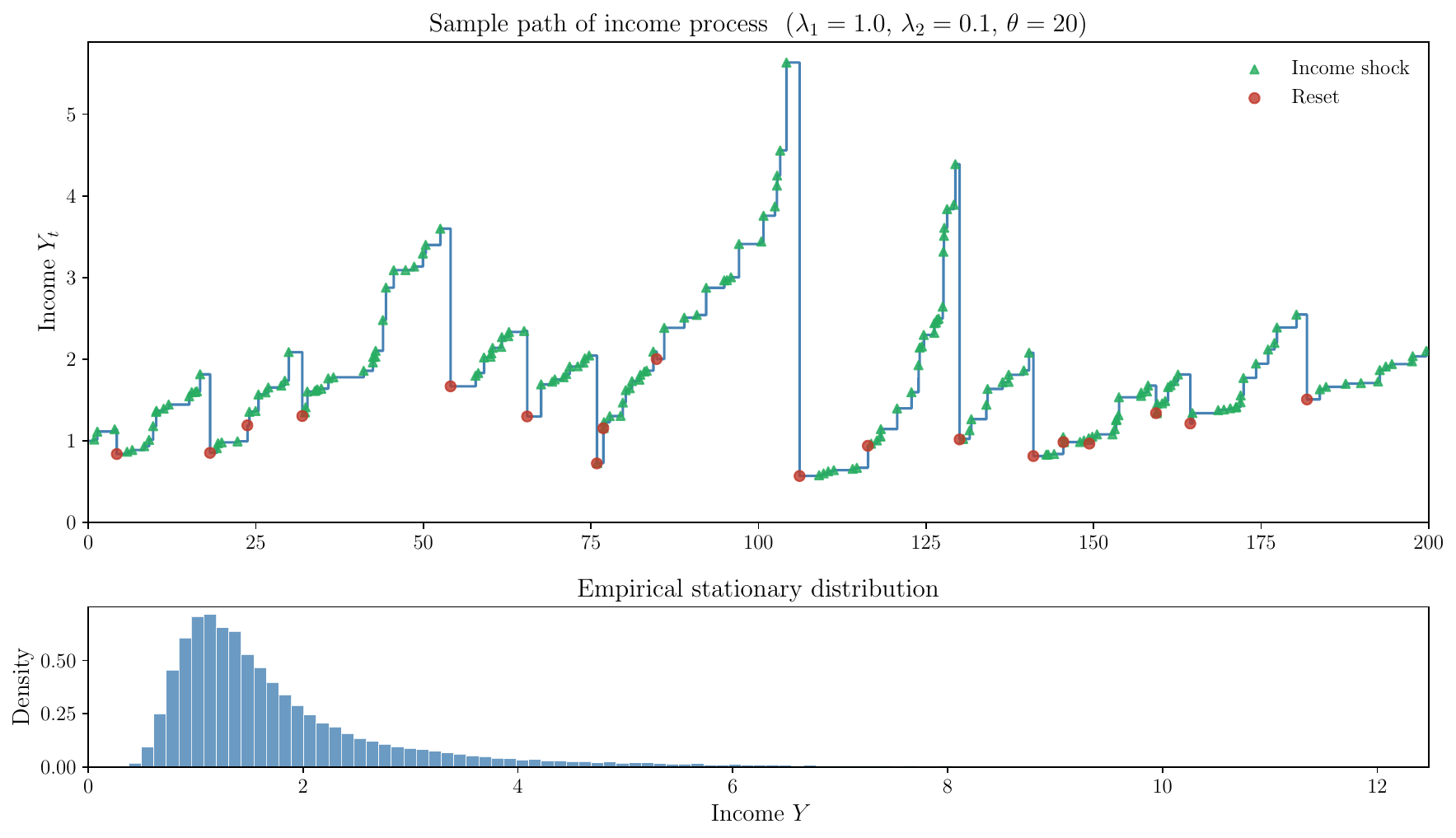}
    \caption{Simulation of the pure jump income model in \cref{ss:purejump}.}
    \label{fig:income_purejump}
\end{figure}

The next result specializes the model to obtain a closed-form stationary
distribution and a Pareto tail.  Income shocks are exponentially distributed and
resets return the agent to a fixed entry level.

\begin{corollary}\label{c:purejumppareto}
    Let $h \equiv x_0$, $\nu = \del_0$, and $\mu =
    \mathrm{Exp}(\theta)$ for some $\theta > 0$. In this setting, for the income
    level $Y = e^X$, the tail of the unique stationary distribution is Pareto:
    for $y > y_0 \coloneq e^{x_0}$,
    \begin{equation}\label{eq:pjpareto}
        \PP(Y > y)
        = q \left(\frac{y}{y_0}\right)^{-\alpha}
        \quad \text{with} \quad
        \alpha \coloneq \frac{\lam_2 \theta}{\lam_1 + \lam_2}.
    \end{equation}
\end{corollary}

\begin{proof}
    Since income is constant between jumps and the jump rate $\lam$ is
    independent of the state, the stationary distribution of the continuous-time
    process $(X_t)$ coincides with that of the embedded chain $(Z_n)$.  The
    chain regenerates at each reset: the post-reset state is always $x_0$.
    Between consecutive resets, the chain accumulates $M$ income shocks before
    the next reset, where $M \sim \mathrm{Geom}(p)$ on $\{0, 1, 2, \ldots\}$.
    In stationarity, the state is $x_0 + S_M$ where $S_M = \eta_1 + \cdots +
    \eta_M$.

    With probability $p$, $M = 0$ and $S_M = 0$, giving an atom at $x_0$.  For
    $M \geq 1$, conditional on $M = m$, the sum $S_m \sim
    \mathrm{Gamma}(m, \theta)$.  Summing over $m$, the density of $S_M$ on
    $(0, \infty)$ is
    \begin{equation*}
        f_{S_M}(s)
        = \sum_{m=1}^{\infty} p q^m
          \cdot \frac{\theta^m s^{m-1}}{(m-1)!} e^{-\theta s}
        = pq\theta \, e^{-\theta s}
          \sum_{k=0}^{\infty} \frac{(q\theta s)^k}{k!}
        = pq\theta \, e^{-p\theta s}
    \end{equation*}
    for $s > 0$.  Since $\phi^*$ is the distribution of $x_0 + S_M$, it follows
    that the density of $\phi^*$ on $(x_0, \infty)$ is 
    $\pi(x) \coloneq pq\theta \, e^{-p\theta(x - x_0)}$. Thus,  for $y \geq y_0$,
    \begin{equation*}
        \PP(Y > y)
        = \int_{\ln y}^{\infty} pq\theta \, e^{-p\theta(x - x_0)} \diff x
        = q \, e^{-p\theta(\ln y - x_0)}
        = q \left(\frac{y}{y_0}\right)^{-\alpha}. \qedhere
    \end{equation*}
\end{proof}

\section{Piecewise Deterministic Markov Processes}\label{s:pdmp}

In this section we develop a general stability result for piecewise
deterministic Markov processes (PDMPs) that can be applied off the shelf
in a range of settings. The theory is presented in \cref{ss:pdmptheory}
and applied to a model of income dynamics with deterministic drift in
\cref{ss:driftmodel}.

\subsection{Theory}\label{ss:pdmptheory}

Let $(S, \leq)$ be a partially ordered Polish space satisfying \cref{a:pos}. A
\navy{piecewise deterministic Markov process} (PDMP) on $S$ is generated by the
following characteristics:
\begin{enumerate}
    \item a \navy{semi-flow} $\Phi \colon S \times \R_+ \to S$,
    \item a \navy{jump intensity} $\lambda \in (0,\infty)$,
    \item a \navy{jump shock space} $U$ and shock distribution $\mu$ on $U$,
        and
    \item a \navy{jump function} $F \colon S \times U \to S$. 
\end{enumerate}
Here $U$ is a Polish space and the jump function $F$ is assumed to be measurable.
The map $\Phi$ is called a \navy{semi-flow} because it satisfies $\Phi(x, 0) = x$
and $\Phi(\Phi(x, s), t) = \Phi(x, s + t)$ for all $x \in S$ and $s, t \geq 0$.
We assume that $\Phi$ is continuous.
Informally, the process evolves as follows:
\begin{enumerate}
    \item Follow the deterministic flow: $X_t = \Phi(X_0, t)$ until the first jump.
    \item At the first jump time $T$, draw $\xi$ from $\mu$ and set $X_T = F(X_{T^-}, \xi)$.
    \item Repeat from the new state $X_T$.
\end{enumerate}
In most applications, the semi-flow is the solution to an ODE.  For example, if
$S = \RR$ and the deterministic dynamics have the linear form $\dot x = a x$, then $\Phi(x, t) = x\exp(at)$.

To formalize the construction of the associated Markov process $(X_t)_{t \geq
0}$, we first collect the following components:
\begin{enumerate}
    \item an $S$-valued random element $X_0$ with distribution $\psi \in \pP(S)$,
    \item a real-valued {\sc iid} sequence $(E_n)_{n \in \NN}$ with common distribution
        Exp$(\lambda)$, and
    \item a $U$-valued {\sc iid} sequence $(\xi_n)_{n \in \NN}$ with common distribution $\mu$,
\end{enumerate}
All random elements live on a common probability space $(\Omega, \fF, \PP)$.
We first build the \navy{embedded chain} $(Z_n)_{n \in \ZZ_+}$ via
\begin{equation*}
    Z_0 = X_0
    \quad \text{and} \quad
    Z_{n+1} = F(\Phi(Z_n, E_{n+1}), \xi_{n + 1}).
\end{equation*}
We define the jump times $T_0 \coloneq 0$ and $T_n \coloneq \sum_{i=1}^n E_i$,
while $N_t \coloneq \max \{n \geq 0 \,:\, T_n \leq t\}$ counts the number of jumps up until time
$t$. Finally, we construct $(X_t)_{t \geq 0}$ via
\begin{equation*}
    X_t = \Phi(Z_{N_t}, t - T_{N_t}) \qquad t \geq 0.
\end{equation*}
A more intuitive representation is
\begin{equation*}
        X_t = \Phi(Z_n, t - T_n)
        \quad \text{for $t \in [T_n, T_{n+1})$ and $n \in \ZZ_+$.}
\end{equation*}
Sample paths are c\`adl\`ag, with smooth deterministic segments broken by
jumps. The process $(X_t)$ obeys the strong Markov property
\citep[Theorem~2.4]{rudnicki2017piecewise}. 
We work with the following assumptions.

\begin{assumption}[PDMP Monotonicity]\label{ass:flow}
    The primitives are such that
    \begin{enumerate}
        \item For all $t \geq 0$, the function $x \mapsto \Phi(x, t)$ is order
            preserving on $S$.  
        \item For all $u \in U$, the function $x \mapsto F(x, u)$ is order preserving on $S$.
    \end{enumerate}
\end{assumption}

In the next assumption, $\xi$ and $\xi'$ are independent draws from $\mu$.

\begin{assumption}[PDMP Mixing]\label{a:pdmpm}
    There exist measurable functions $f_1,
    f_2 \colon U \to S$ such that
    \begin{equation*}
        \text{
            $f_1(\xi) \leq F(x, \xi) \leq f_2(\xi)$
            for all $x \in S$ and $\PP\{f_2(\xi) \leq f_1(\xi')\} > 0$.
        }
    \end{equation*}
\end{assumption}

The transition probability function corresponding to the PDMP $(X_t)_{t \geq 0}$
is given by 
\begin{equation*}
    P_t \,(x, B) = \PP \{ X_t \in B \,|\, X_0 = x\}
    \qquad (t \in \RR_+, \; x \in S, \; B \in \bB).
\end{equation*}
Under the stated monotonicity and mixing conditions, we can prove
the following result.  

\begin{theorem}\label{t:pdmpgs}
    If \cref{ass:flow,a:pdmpm} hold, then the PDMP is globally stable, with a
    unique stationary distribution $\phi^*$ and
    \begin{equation}\label{eq:pdmps}
        \lim_{t \to \infty} \beta(\phi P_t, \phi^*) = 0
        \quad \text{for all} \quad \phi \in \pP(S).
    \end{equation}
\end{theorem}

The proof is based on \cref{t:orgsnew} and can be found in
\cref{ss:pdmpgsproof}.  The strategy is to verify the three hypotheses of
\cref{t:orgsnew}.  Monotonicity is established by coupling two copies of the
PDMP with the same jump times and the same shocks: if one copy starts above the
other, then \cref{ass:flow} ensures that it stays above through both the
deterministic flow and the jumps.  Boundedness in probability follows from the
pathwise bounds in \cref{a:pdmpm}, which prevent post-jump states from escaping
to infinity, while the exponential inter-arrival times prevent the deterministic
flow from carrying the process too far between jumps.  The most interesting step
is weak order mixing, which uses a different coupling: the two copies share
jump times but receive independent shocks.  The condition
$\PP\{f_2(\xi) \leq f_1(\xi')\} > 0$ in \cref{a:pdmpm} then guarantees that,
at each shared jump, there is a positive probability that the relative ordering
of the two processes is reversed, regardless of their current states.

Now we turn to applications.

\subsection{Application: Income Dynamics Part II}\label{ss:driftmodel}

In Section~\ref{ss:purejump} we analyzed income dynamics through a pure jump
process. We now consider a model of income dynamics more closely aligned with \citet{gabaix2016dynamics},
where log income has a continuous drift component.  Unlike
\cite{gabaix2016dynamics}, we continue to avoid
the assumption of positive quadratic variation (e.g., Brownian motion), which is
difficult to motivate from data.  To make the proofs slightly simpler, we drop
the positive income shocks in Section~\ref{ss:purejump} (while retaining the
resets).

In our model, log income drifts deterministically between jumps,
following the ODE $\dot x = g(x)$ for some locally Lipschitz function $g \colon
\RR \to \RR$. At jump times, which occur at rate $\lambda$,
the state jumps from $x$ to $F(x, \zeta)$, where $\zeta$ is an independent draw
from $\nu$.

More formally, we study a PDMP on $S = \RR$ with the following primitives:
\begin{enumerate}
    \item a semi-flow $\Phi \colon \RR \times \RR_+ \to \RR$ generated by the ODE
        $\dot x = g(x)$,
    \item a constant jump intensity $\lam > 0$,
    \item a shock distribution $\nu$ on $\RR$, and
    \item a jump function $F(x, z) = h(x) + z$, where $h \in ibS$.
\end{enumerate}
Between jumps, $X_t = \Phi(Z_{N_t}, t - T_{N_t})$, where $(E_n)_{n \in \NN}$
are {\sc iid} Exp$(\lam)$ inter-arrival times and $N_t$ counts jumps up to time $t$
(as in \cref{s:pdmp}).  At each jump time, the state moves to
$h(X_{T^-}) + \zeta$.  The embedded chain is
$Z_{n+1} = h(\Phi(Z_n, E_{n+1})) + \zeta_{n+1}$.

The next proposition shows global stability under essentially the same
assumptions as \cref{p:purejumpstable}.

\begin{proposition}\label{p:driftincomestable}
    If the support of $\nu$ is larger than the range of $h$, then the log income
    process is globally stable: there exists a unique stationary distribution
    $\phi^*$ and
    \begin{equation}\label{eq:driftgs}
        \lim_{t \to \infty} \beta(\phi P_t, \phi^*) = 0
        \quad \text{for all} \quad \phi \in \pP(\RR).
    \end{equation}
\end{proposition}

\begin{proof}
    We verify \cref{ass:flow,a:pdmpm}. Then we apply \cref{t:pdmpgs}.

    \textit{\cref{ass:flow}\,(i): the semiflow is order preserving.}  Since $g$
    is locally Lipschitz, the Picard--Lindel\"of theorem gives uniqueness of
    solutions.  Fix $x_1 < x_2$ and $t > 0$.  Seeking a contradiction, suppose 
    that $\Phi(x_1, s) \geq \Phi(x_2, s)$ for some $s > 0$.  Since
    $\Phi(x_1, 0) < \Phi(x_2, 0)$ and $\Phi$ is continuous, the intermediate
    value theorem yields $t_0 \in (0, s]$ with $\Phi(x_1, t_0) = \Phi(x_2,
    t_0)$.  Both trajectories pass through the same state at time $t_0$ and the
    ODE is autonomous, so uniqueness forces them to agree for all time,
    contradicting $x_1 \neq x_2$.  Hence $x \mapsto \Phi(x, t)$ is strictly
    increasing for every $t \geq 0$.

    \textit{\cref{ass:flow}\,(ii): the jump function is order preserving.}  For
    each $\zeta \in \RR$, the map $x \mapsto F(x, \zeta) = h(x) + \zeta$ is
    increasing since $h$ is assumed to be increasing.

    \textit{\cref{a:pdmpm}: bounding functions and coupling.}  Set $f_1(\zeta)
    = \underline{h} + \zeta$ and $f_2(\zeta) = \bar{h} + \zeta$, where
    $\underline{h} \coloneq \inf h$ and $\bar{h} \coloneq \sup h$.  By
    construction, $f_1(\zeta) \leq F(x, \zeta) \leq f_2(\zeta)$ for
    all $x \in \RR$.  For the coupling condition, let $\zeta, \zeta'$ be
    independent draws from $\nu$.  Then
    \begin{equation*}
        \PP\{f_2(\zeta) \leq f_1(\zeta')\}
        = \PP\{\zeta' - \zeta \geq \bar{h} - \underline{h}\}
        > 0,
    \end{equation*}
    where the strict inequality follows from our assumption on the support of
    $\nu$.  Global stability now follows from \cref{t:pdmpgs}.
\end{proof}

If we specialize to constant drift rate $g$ and a pure reset, we can easily obtain a
Pareto tail, in a result analogous to \cref{c:purejumppareto}. The tail exponent
is $\lam/g$, and hence entirely determined by the first-order balance between
deterministic growth and random resetting.  Further details are available from
the authors on request.

\cref{fig:income_pdmp} illustrates a simulation with constant drift $g(x) =
\mu = 0.05$, jump rate $\lam = 0.15$, pure reset $h \equiv 0$, and {\sc iid}
Gaussian reset shocks $\zeta \sim N(0, 0.09)$. The time series of income $Y_t
= e^{X_t}$ displays characteristic PDMP dynamics: smooth exponential growth
between jumps, interrupted by sharp resets. The empirical stationary
distribution is right-skewed with a Pareto tail of exponent $\lam/\mu
= 3$.

\begin{figure}[ht]
    \centering
    \includegraphics[width=\textwidth]{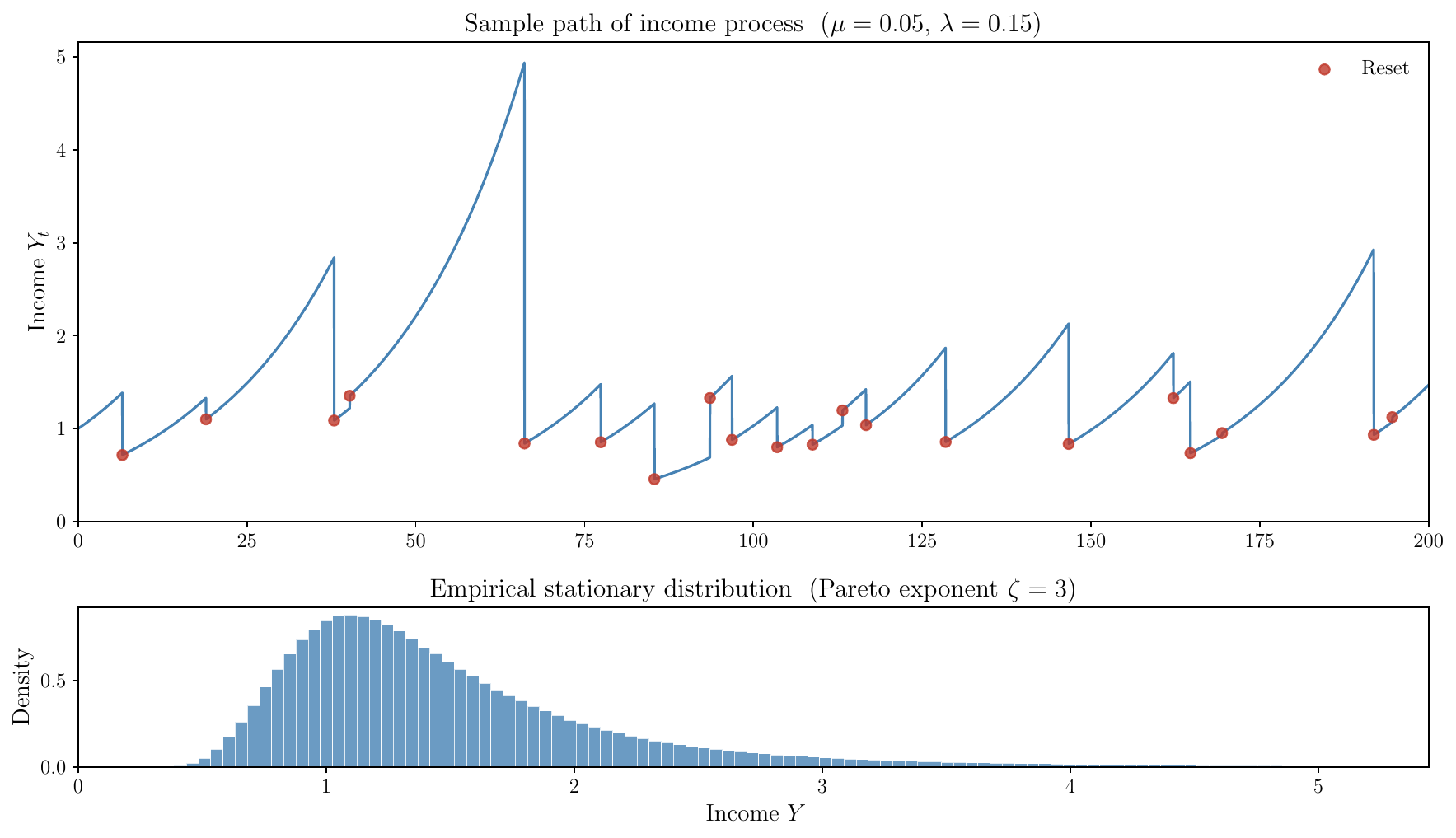}
    \caption{Simulation of the income dynamics model in \cref{ss:driftmodel}.}
    \label{fig:income_pdmp}
\end{figure}

\section{Conclusion}

We have shown that an order-preserving semigroup on a complete preordered
metric space is globally stable if and only if it is asymptotically
contractive and has at least one order-bounded trajectory. When applied to
monotone Markov models on partially ordered Polish spaces, this yields a
necessary and sufficient condition for the existence, uniqueness, and global
stability of stationary distributions: asymptotic contractivity combined
with tightness of at least one trajectory.

These results unify and extend earlier work by \cite{hopenhayn1992stochastic},
\cite{kamihigashi2014stochastic}, and other authors, weakening assumptions,
introducing new methods, and removing extraneous side conditions. The abstract
semigroup framework covers both discrete and continuous time, as well as
nonlinear Markov operators. Our applications to wage dynamics, Bayesian learning
with belief shocks, and income dynamics with Pareto tails illustrate the breadth
of the approach.

Many directions remain open. The mixing conditions used in this paper could
potentially be weakened further, for example by allowing state-dependent jump
intensities in the PDMP framework. Our Pareto tail result for income is
specialized and may hold under weaker assumptions.  The relationship between
tightness and order boundedness established here may also prove useful in other
settings involving monotone economic dynamics.

\appendix

\section{Proofs}\label{s:proofs}

This appendix collects remaining proofs.

\subsection{Proof of \texorpdfstring{\cref{t:bk}}{Theorem \ref{t:bk}}}

Throughout this section, \cref{a:env,a:env2} are in force.

\subsubsection{(ii) implies (i)}

Let $(T_t)$ be globally stable with stationary point $x^* \in
\Xsf$.  Then the sequence given by $T_t x^{*}$ for all $t \in \TT$ is an
order-bounded trajectory of $(T_t)$.  Moreover, picking arbitrary $x, y \in \Xsf$
and applying the triangle inequality,
\begin{equation*}
    d(T_t x, T_t y) \leq d(T_t x, x^*) + d(x^*, T_t y) \rightarrow 0
    \qquad (t \to \infty),
\end{equation*}
where the convergence holds by global stability.  Hence $(T_t)$ is
asymptotically contractive.

\subsubsection{(i) implies (ii)}

Now we assume that $(T_t)$ is asymptotically contractive on $\Xsf$ with at least one order-bounded trajectory.
Due to the last property, we can take $a, b, \bar x \in \Xsf$ such that
\eqref{eq:bt} holds.

\begin{lemma}\label{l:dis}
    For all $\epsilon > 0$, there exists a $\tau \in \TT$ such that
    \begin{equation}\label{eq:dis}
        s,t \in \TT \text{ with }  \tau \leq s,t
        \quad \implies \quad
        d(T_s \, \bar x, T_t \, \bar x) < \epsilon .
    \end{equation}
\end{lemma}

\begin{proof}
    Fix $\epsilon > 0$ and then use asymptotic
    contractivity to choose $\tau \in \TT$ such that $d(T_\tau \, a, T_\tau \, b) <
    \epsilon$.  For $s, t \geq \tau$ we have
    \begin{equation*}
        d(T_s \, \bar x, T_t \, \bar x)
         = d(T_\tau  T_{s - \tau} \, \bar x, T_\tau T_{t - \tau} \, \bar x)
         \leq d(T_\tau \, a, T_\tau \, b) < \epsilon.
    \end{equation*}
    The first inequality is due to the diagonal property of $d$, combined with
    the fact that $T_{s - \tau} \, \bar x$ and $T_{t - \tau} \, \bar x$ both lie in $[a, b]$,
    which in turn implies that
    \begin{equation*}
        T_\tau \, a \preceq T_\tau T_{s-\tau} \,  \bar x \preceq T_\tau\, b
        \quad \text{and} \quad
        T_\tau \, a \preceq T_\tau T_{t-\tau} \, \bar x \preceq T_\tau \, b.
        \qedhere
    \end{equation*}
\end{proof}

\begin{lemma}\label{l:conv}
    There exists an $x^* \in \Xsf$ such that
    \begin{equation}\label{eq:stall}
         d(T_t \, \bar x, x^*) \to 0
         \quad \text{as} \quad
         t \to \infty.
    \end{equation}
\end{lemma}

\begin{proof}
    Let $(t_n)_{n \in \NN}$ be an increasing sequence in $\TT$ with $t_n \to \infty$ as $n \to \infty$.
    Consider the sequence defined by $x_n \coloneq T_{t_n} \bar x$.  This sequence is
    Cauchy because, given $\epsilon > 0$, we can choose $\tau \in \TT$ such that
    \eqref{eq:dis} holds and then $d(x_n, x_m) < \epsilon$ for all $n, m$ such
    that $t_n, t_m \geq \tau$.  Since $\Xsf$ is complete, it follows that $(x_n)$
    converges to a limit $x^*$.  That is,
    \begin{equation}\label{eq:cass}
         d(T_{t_n} \, \bar x, x^*) \to 0
         \quad \text{as} \quad
         n \to \infty.
    \end{equation}
    We claim that this can be extended to \eqref{eq:stall}.
    To see this, fix $\epsilon > 0$ and choose $\tau$ as in \eqref{eq:dis}. For $t
    \geq \tau$ and $n \in \NN$ such that $t_n \geq \tau$, we have
    \begin{equation*}
        d(T_t \, \bar x, x^*)
        \leq d(T_t \, \bar x, T_{t_n} \, \bar x) + d(T_{t_n} \, \bar x, x^*)
         < \epsilon + d(T_{t_n} \, \bar x, x^*)
    \end{equation*}
    Applying \eqref{eq:cass} and taking the limit in $n$ completes the argument.
\end{proof}

\begin{lemma}\label{l:fpfxs}
    The point $x^*$ in \cref{l:conv} is stationary for $(T_t)$.
\end{lemma}

\begin{proof}
    Fix $s, t \in \TT$. From the existence of an order-bounded trajectory we have $a \preceq
    T_t \, \bar x \preceq b$. From the order-preserving property
    of $T_s$ we get $T_s \, a \preceq T_{t + s} \, \bar x \preceq T_s \, b$.
    Since $\preceq$ is closed, taking the limit with respect to $t$ and applying \cref{l:conv} yields
    \begin{equation}\label{eq:sta}
        T_s \, a \preceq x^* \preceq T_s \, b
        \quad \text{for all } \; s \in \TT.
    \end{equation}
    In addition, from \eqref{eq:sta} and the order-preserving property
    of $T_t$,
    \begin{equation}\label{eq:stab}
        T_{u + t} \, a \preceq T_t \, x^* \preceq T_{u + t} \, b
        \quad \text{for all } \; u, t \in \TT.
    \end{equation}
    Combining \eqref{eq:sta} and \eqref{eq:stab} with the diagonal property
    \eqref{eq:diag}, we get
    \begin{equation}\label{eq:fxss}
        d(x^*, T_t \, x^*) \leq d(T_{u+t} \, a, T_{u+t} \, b)
        \quad \text{for all } \; u \in \TT.
    \end{equation}
    Taking the limit in $u$ and using asymptotic contractivity yields
    $T_t \, x^* = x^*$.
\end{proof}

\begin{proof}[Proof of \cref{t:bk}]
    We have already shown that $x^*$ is stationary for $(T_t)$ in
    \cref{l:fpfxs}.  Regarding convergence, fix $x \in \Xsf$.  Since $(T_t)$
    is asymptotically contractive and $x^*$ is stationary, we have
    $d(T_t \, x, x^*) = d(T_t \, x, T_t \, x^*) \to 0$ as $t \to \infty$.
    This concludes the proof of \cref{t:bk}.
\end{proof}

\subsection{Proof of \texorpdfstring{\cref{t:bk2}}{Theorem \ref{t:bk2}}}

We only used closedness of $\preceq$ in \cref{l:fpfxs}, when we
proved that the limit point $x^*$ of the order-bounded trajectory $(T_t \,
\bar x)$ is a fixed point of each $T_t$.  We now prove this without
closedness by fixing $t \in \TT$ and applying the triangle inequality to
obtain
\begin{equation*}
    d(x^*, T_t \, x^*) \leq d(x^*, T_{u+t} \, \bar x) + d(T_{u+t} \, \bar x, T_t
    \, x^*).
\end{equation*}
The first term converges to zero in $u$ by \cref{l:conv}.  If we can
show that the second term also converges to zero in $u$ then we are done.
This holds because $T_t$ is continuous and $T_u \, \bar x \to x^*$, so, as $u
\to \infty$,
\begin{equation*}
    d(T_{u+t} \, \bar x, T_t \, x^*)
    = d(T_t T_u \, \bar x, T_t \, x^*)
    \to
    d(T_t \, x^*, T_t \, x^*) = 0.
\end{equation*}

\subsection{Proof of \texorpdfstring{\cref{t:act}}{Theorem \ref{t:act}}}

Let the conditions of \cref{t:act} hold. Because $S$ is a
partially ordered Polish space and \cref{a:pos} is in force, $(\pP(S), \beta)$
is a complete metric space \citep[Theorem~4.1]{kamihigashi2019unified}
and $\precsd$ is closed \citep{kamae1978stochastic}.
We begin with the following lemma.

\begin{lemma}\label{l:tt}
    If $\Lambda \subset \pP(S)$ is tight, then there exists an increasing sequence of order
    intervals $(I_n)_{n \in \Z_+}$ such that
    $\sup_{\phi \in \Lambda} \phi(I_n^c) \to 0$ as $n \to \infty$.
\end{lemma}

\begin{proof}
    Let $(\epsilon_n)$ be a positive real sequence with $\epsilon_n \downarrow
    0$.  By tightness, we can choose compact set $K_0$ with $\sup_{\phi \in
    \Lambda} \phi(K_0^c) \leq \epsilon_0$.  By \cref{a:pos}, compact sets are order
    bounded, so we
    can also take an order interval $I_0$ containing $K_0$ with $\sup_{\phi \in
    \Lambda} \phi(I_0^c) \leq \epsilon_0$. Next, we choose a compact set $K_1$ with
    $\sup_{\phi \in \Lambda} \phi(K_1^c) \leq \epsilon_1$, and then, using
    \cref{a:pos} again, an order interval $I_1$ such that $I_0 \cup
    K_1 \subset I_1$.  This order interval obeys $I_0 \subset I_1$ and
    $\sup_{\phi \in \Lambda} \phi(I_1^c) \leq \epsilon_1$.  Continuing in this way
    produces an increasing sequence of order intervals $(I_n)_{n \in \Z_+}$
    with $\sup_{\phi \in \Lambda} \phi(I_n^c) \to 0$.
\end{proof}

\begin{proof}[Proof of \cref{p:tiob}]
    Suppose first that $\Lambda$ is tight.  We show that there exists a pair of
    distributions $\ell, u \in \pP(S)$ such that $\ell \preceq \phi \preceq u$ for
    all $\phi \in \Lambda$.   By \cref{l:tt}, we can take an increasing sequence of order
    intervals $(I_n)_{n \in \Z_+}$ and a positive real sequence
    $(\epsilon_n)$ with $\epsilon_n \downarrow 0$  and $\sup_{\phi \in \Lambda} \phi(I_n^c)
    \leq \epsilon_n$ for all $n \in \Z_+$.
    Let $I_n = [x_n, y_n]$. Since $(I_n)$ is increasing, the sequence
    $(x_n)$ is decreasing and $(y_n)$ is increasing.
    For the distributions $\ell$ and $u$, we set
    \begin{equation*}
        \ell \coloneq \sum_{n=0}^\infty \alpha_n \delta_{x_n}
        \quad \text{and} \quad
        u \coloneq \sum_{n=0}^\infty \alpha_n \delta_{y_n},
        \quad \text{where} \quad
        \alpha_n \coloneq \epsilon_{n-1} - \epsilon_n.
    \end{equation*}
    In the definition above, $\epsilon_{-1} \coloneq 1$ and $\delta_z$ is the
    probability measure concentrated on $z$.  One easily confirms that both
    $\ell$ and $u$ are elements of $\pP(S)$.

    Now let $J$ be any increasing set in $\bB$. We claim that $\phi(J) \leq u(J)$ for all
    $\phi \in \Lambda$.  To see that this is so, fix $\phi \in \Lambda$ and suppose
    first that $y_m \in J^c$ for all $m$.  If this is so, then $\phi(J) = 0$.
    Indeed, fixing $m \in \Z_+$, the property $y_m \in J^c$ implies $J \subset
    I^c_m$. (Otherwise there is an $x \in J \cap I_m$ and, since $J$ is
    increasing and $x \preceq y_m$, we have $y_m \in J$.)  This means
    that $\phi(J) \leq \phi (I^c_m) \leq \epsilon_m$.  Taking the limit in $m$
    yields $\phi(J) = 0$.  Hence $\phi(J) \leq u(J)$.

    Now consider the other case, where there exists an integer $m$ satisfying
    \begin{equation*}
        m = \min\{n \in \Z_+ \, : \, y_n \in J\}.
    \end{equation*}
    Observe that $J \cap I_{m-1}$ is empty. (If $x$ is in both, then,
    since $x \leq y_{m-1}$ and $x \in J$, we have $y_{m-1} \in J$, which
    contradicts the definition of $m$.)  As a result,
    \begin{equation*}
        \phi(J)
        = \phi(J \cap I_{m-1}) + \phi(J \cap I^c_{m-1})
        = \phi(J \cap I^c_{m-1})
        \leq \phi(I^c_{m-1})
        \leq \epsilon_{m-1}.
    \end{equation*}
    At the same time, since $(y_n)$ is increasing,
    \begin{equation*}
        u(J)
        = \epsilon_{m-1} - \epsilon_m
            + \epsilon_m - \epsilon_{m+1}
            + \epsilon_{m+1} - \epsilon_{m+2}
            + \cdots
        = \epsilon_{m-1}.
    \end{equation*}
    As a consequence, we have $\phi(J) \leq u(J)$.  This shows that $\phi
    \precsd u$.  A similar argument shows that $\ell \precsd \phi$ also holds.
    Hence $\Lambda$ is order bounded.

    Conversely, suppose $\Lambda$ is order bounded, so there exist $\ell, u \in
    \pP(S)$ with $\ell \precsd \phi \precsd u$ for all $\phi \in \Lambda$.  Fix
    $\epsilon > 0$.  Since $\ell$ and $u$ are distributions on a Polish space,
    they are tight, so we can choose compact sets $K_\ell$ and $K_u$ with
    $\ell(K_\ell^c) \leq \epsilon/2$ and $u(K_u^c) \leq \epsilon/2$.  By
    \cref{a:pos}, the compact set $K_\ell \cup K_u$ is order bounded, so there
    exist $a, b \in S$ with $K_\ell \cup K_u \subset [a, b]$.  By \cref{a:pos}
    again, $[a, b]$ is compact.

    Let $U_a \coloneq \{x \in S : x \geq a\}$ and $D_b \coloneq \{x \in S : x
    \leq b\}$.  Then $[a,b] = U_a \cap D_b$ and $[a,b]^c = U_a^c \cup D_b^c$.
    Fix $\phi \in \Lambda$.  Since $U_a$ is increasing, we have $\ell(U_a) \leq
    \phi(U_a)$, giving $\phi(U_a^c) \leq \ell(U_a^c)$.  Similarly, since
    $D_b^c$ is increasing and $\phi \precsd u$, we have $\phi(D_b^c) \leq
    u(D_b^c)$.  Now $K_\ell \subset [a,b]$ implies $U_a^c \subset [a,b]^c
    \subset K_\ell^c$, so $\ell(U_a^c) \leq \epsilon/2$.  Likewise, $u(D_b^c)
    \leq \epsilon/2$.  Therefore,
    \begin{equation*}
        \phi([a,b]^c)
        \leq \phi(U_a^c) + \phi(D_b^c)
        \leq \ell(U_a^c) + u(D_b^c)
        \leq \epsilon.
    \end{equation*}
    Since $[a,b]$ is compact and this holds for all $\phi \in \Lambda$, we conclude
    that $\Lambda$ is tight.
\end{proof}

\begin{proof}[Proof of \cref{t:act}]
    Let the conditions of \cref{t:act} hold.

    ((ii) $\Leftrightarrow$ (iii)) This equivalence is immediate from \cref{p:tiob}.

    ((i) $\implies$ (ii))  If $(T_t)$ is globally stable, then $(T_t)$ is asymptotically
    contractive by \cref{t:bk}.  Moreover, when $\psi^*$ is the unique
    stationary distribution, the constant trajectory $(T_t \psi^*) = (\psi^*)$
    is trivially order bounded.

    ((ii) $\implies$ (i))
    Suppose that $(T_t)$ is asymptotically contractive and has an order-bounded trajectory.
    The space $(\pP(S), \beta)$ is complete and $\beta$ satisfies the diagonal
    property by \cref{l:mdi}.  Moreover, $\precsd$ is closed with respect to
    $\beta$ by \cref{l:betaclosed}.  Hence $(T_t)$ is globally stable by \cref{t:bk}.
\end{proof}

\subsection{Proof of \texorpdfstring{\cref{t:hp}}{Theorem \ref{t:hp}}}\label{ss:proofhp}

Let the conditions of \cref{t:hp} hold. We write $\lfloor \cdot \rfloor$
for the floor function.

\begin{lemma}\label{lem:bounds}
    If the conditions of \cref{t:hp} hold, then
    \begin{equation}
        (1-\epsilon)\delta_a + \epsilon \delta_{\bar{x}} \precsd P_u(a, \cdot)
        \quad \text{and} \quad
        P_u(b, \cdot) \precsd (1-\epsilon)\delta_b + \epsilon \delta_{\bar{x}}
    \end{equation}
\end{lemma}

\begin{proof}
    See the proof of Eq.~(2) in Theorem~2 of \cite{hopenhayn1992stochastic}.
\end{proof}

\begin{lemma}\label{lem:contraction}
    If the conditions of \cref{t:hp} hold, then
    \begin{equation}
        \beta(\phi P_u^n, \psi P_u^n) \leq 2(1-\epsilon)^n
        \quad \text{for all  $\phi, \psi \in \pP(S)$ and $n \in \mathbb{Z}_+$.}
    \end{equation}
\end{lemma}

\begin{proof}
    We set $P = P_u$ for brevity.
    Fix $h \in ibS$ with $|h| \leq 1$. Define $d_n = P^n h(a)$ and $e_n = P^n h(b)$.
    Since $P$ is increasing and $h$ is increasing, $P^{n-1}h$ is also increasing for all $n \geq 1$.
    Applying the bounds from Lemma \ref{lem:bounds} to the increasing function
    $P^{n-1}h$ yields
    \begin{align*}
        d_n &= (P (P^{n-1}h)) (a) \geq (1-\epsilon)d_{n-1} +
            \epsilon (P^{n-1}h)(\bar{x}) \\
        e_n &= (P( P^{n-1}h))(b) \leq (1-\epsilon)e_{n-1} + \epsilon (P^{n-1}h)(\bar{x})
    \end{align*}
    Subtracting one inequality from the other, we get $e_n - d_n \leq
    (1-\epsilon)(e_{n-1} - d_{n-1})$, and hence
    \begin{equation}
        e_n - d_n \leq (1-\epsilon)^n(e_0 - d_0) =
        (1-\epsilon)^n(h(b) - h(a)) \leq 2(1-\epsilon)^n.
    \end{equation}
    Since $a$ is the bottom and $b$ is the top of $S$, we have $(P^n h)(a) \leq
    (P^n h)(x) \leq (P^n h)(b)$ for all $x \in S$, so, integrating with respect
    to $\phi \in \pP(S)$,
    \begin{equation*}
        (P^n h)(a) \leq \phi (P^n h) \leq (P^n h)(b).
    \end{equation*}
    Therefore, for any two probability measures $\phi$ and $\psi$, we have
    \begin{equation*}
        |\phi (P^n h) - \psi (P^n h) | \leq (P^n h)(b) - (P^n h)(a) \leq 2(1-\epsilon)^n.
    \end{equation*}
    Taking the supremum over all $h \in ibS$ with $|h| \leq 1$ proves the claim in the lemma.
\end{proof}

\begin{proof}[Proof of \cref{t:hp}]
    Let the stated conditions hold. We first claim that $(P_t)_{t \in \TT}$ is
    asymptotically contractive. For $\TT = [0, \infty)$, fix
    $t \in (0, \infty)$, write $t = nu + r$ where $n = \lfloor t/u \rfloor$ and
    $0 \leq r < u$. By the semigroup property, we have
    $P_t = P_{nu + r} = P_{nu} P_r = P_u^n P_r$, so
    \begin{equation*}
        \beta(\phi P_t, \psi P_t)
        = \beta(\phi P_u^n P_r, \psi P_u^n P_r)
        \leq \beta(\phi P_u^n, \psi P_u^n),
    \end{equation*}
    where the inequality is by the nonexpansiveness in \cref{lem:nonexp}.
    Applying the contraction bound in \cref{lem:contraction}, we obtain
    \begin{equation*}
        \beta(\phi P_u^n, \psi P_u^n)
        \leq 2(1-\epsilon)^n
        = 2(1-\epsilon)^{\lfloor t/u \rfloor}.
    \end{equation*}
    As $t \to \infty$, $\lfloor t/u \rfloor \to \infty$, so $(P_t)_{t \geq 0}$
    is asymptotically contractive, as claimed.

    For $\TT = \ZZ_+$, the proof is similar.  We write $t = nu + r$ where
    $n = \lfloor t/u \rfloor$ and $r$ is an integer obeying $0 \leq r < u$.
    Since $P_t = P_u^n P_r$, we can again use \cref{lem:nonexp,lem:contraction}
    to obtain
    \begin{equation*}
    \beta(\phi P_t, \psi P_t) \leq \beta(\phi P_u^n, \psi P_u^n) \leq 2(1-\epsilon)^n \to 0
    \end{equation*}
    as $t \to \infty$. Hence $(P_t)_{t \in \ZZ_+}$ is asymptotically contractive.

    To complete the proof of global stability, we use \cref{t:act}. For any
    increasing $h \in ibS$ and any $\phi \in \pP(S)$, we have $h(a) \leq \phi(h)
    \leq h(b)$, and hence $\delta_a \precsd \phi \precsd \delta_b$. In
    particular, for any $\phi \in \pP(S)$ and any $t \in \TT$, we have $\delta_a
    \precsd \phi P_t \precsd \delta_b$. Thus every trajectory of $(P_t)$ is
    order bounded. Since $(P_t)$ is asymptotically contractive, increasing, and
    has an order-bounded trajectory, \cref{t:act} implies that $(P_t)$ is
    globally stable.

    For the exponential bound \eqref{eq:exp_conv}, let $\phi^*$ be the unique
    fixed point. Since $\lfloor t/u \rfloor \geq t/u - 1$, we have
    $(1-\epsilon)^{\lfloor t/u \rfloor} \leq (1-\epsilon)^{t/u - 1}$.
    Writing $(1-\epsilon)^{t/u} = e^{-\alpha t}$ where $\alpha = \ln(1/(1-\epsilon))/u > 0$,
    we obtain
    \begin{equation*}
        \beta(\phi P_t, \psi P_t) \leq \frac{2}{1-\epsilon} e^{-\alpha t}.
    \end{equation*}
    Setting $\psi = \phi^*$ and using $\phi^* P_t = \phi^*$ yields the bound in
    \eqref{eq:exp_conv}.

    The monotone ergodicity result in \cref{t:hp} follows from Theorem~3.1 and
    Proposition~4.1 of \cite{kamihigashi2016seeking}.
\end{proof}

\subsection{Order-Theoretic Mixing and Asymptotic Contractivity}

In this section we explore the connection between order mixing 
properties and asymptotic contractivity in the Bhattacharya metric.
Here is one key result.

\begin{proposition}\label{p:orwomnew}
    Let $(P_t)$ be an increasing transition probability function on $S$.
    If $(P_t)$ is weakly order mixing, then $(P_t)$ is
    asymptotically contractive on $\pP(S)$ with respect to the metric $\beta$.
\end{proposition}

\begin{proof}[Proof of \cref{p:orwomnew}]
    Fix $x, x' \in S$ and $h \in ibS$ with $|h| \leq 1$.  Let $((X_t),
    (X_t'))$ be as in the definition of weakly order mixing and $\tau$ be as
    defined in \eqref{eq:omb}. We have
    \begin{equation*}
        \delta_x P_t h - \delta_{x'} P_t h
        = \EE_x  h(X_t) - \EE_{x'} h(X_t')
        = \EE_x \EE_\tau h(X_t) - \EE_{x'} \EE_\tau h(X_t')
    \end{equation*}
    By construction, $(X_t)$ and $(X_t')$ are defined on the same probability
    space, so we can rewrite this as
    \begin{equation*}
        \delta_x P_t h - \delta_{x'} P_t h
        = \EE_{x,x'} [\EE_\tau h(X_t) - \EE_\tau h(X_t')]
        = A + B,
    \end{equation*}
    where
    \begin{align*}
        A & \coloneq \EE_{x,x'} [\EE_\tau h(X_t) - \EE_\tau h(X_t')] \{\tau \leq t\}
        \\
        B & \coloneq \EE_{x,x'} [\EE_\tau h(X_t) - \EE_\tau h(X_t')] \{\tau > t\}.
    \end{align*}
    Using the strong Markov property and the fact that $\tau$ is a finite
    stopping time, we can write $A$ as
    \begin{equation*}
        A = \EE_{x,x'} [(P_{t-\tau}h)(X_\tau) - (P_{t-\tau} h)(X_\tau')] \{\tau \leq t\}
    \end{equation*}
    Since $h$ is increasing, $X_\tau \leq X_\tau'$ and $(P_t)$ is increasing, we
    must have $A \leq 0$.  Moreover,
    \begin{equation*}
        B \leq \EE_{x,x'} |\EE_\tau h(X_t) - \EE_\tau h(X_t')| \{\tau > t\}
        \leq 2 \PP_{x,x'} \{\tau > t\}.
    \end{equation*}
    As a result,
    \begin{equation*}
        \delta_x P_t h - \delta_{x'} P_t h \leq 2 \PP_{x,x'} \{\tau > t\}.
    \end{equation*}
    Since the one-sided bound holds for all $x, x' \in S$, applying it
    to the pair $(x', x)$ gives $\delta_{x'} P_t h - \delta_x P_t h
    \leq 2 \PP_{x',x} \{\tau > t\}$, and hence
    \begin{equation*}
        |\delta_x P_t h - \delta_{x'} P_t h| \leq 2 \max(\PP_{x,x'} \{\tau > t\},\, \PP_{x',x} \{\tau > t\}).
    \end{equation*}
    Since $h \in ibS$ with $|h| \leq 1$ was arbitrary, we can take the supremum to get
    \begin{equation}\label{eq:acd}
        \beta(\delta_x P_t , \delta_{x'} P_t) \leq 2 \max(\PP_{x,x'} \{\tau > t\},\, \PP_{x',x} \{\tau > t\})
        \quad \text{for all } x,x' \in S.
    \end{equation}
    We wish to extend this to arbitrary (nondegenerate) initial conditions.
    To this end, fix $\phi, \psi \in \pP(S)$. Let $(X, X')$ be any coupling with marginals
    $\phi$ and $\psi$. For $h \in ibS$ with $|h| \leq 1$, we have
    \begin{equation*}
        (\phi P_t)(h) - (\psi P_t)(h)
        = \EE[(P_t h)(X)] - \EE[(P_t h)(X')]
        = \EE[(\delta_X P_t)(h) - (\delta_{X'} P_t)(h)].
    \end{equation*}
    Taking absolute values and using \eqref{eq:acd}, we get
    \begin{equation*}
        |(\phi P_t)(h) - (\psi P_t)(h)| \leq \EE[\beta(\delta_X P_t, \delta_{X'} P_t)].
    \end{equation*}
    The right-hand side does not depend on $h$, so taking the supremum over $h \in ibS$ with $|h| \leq 1$ gives
    \begin{equation*}
        \beta(\phi P_t, \psi P_t) \leq \EE[\beta(\delta_X P_t, \delta_{X'} P_t)].
    \end{equation*}
    By \eqref{eq:acd},
    \begin{equation}\label{eq:wombd}
        \beta(\phi P_t, \psi P_t) \leq 2\EE[\max(\PP_{X,X'}\{\tau > t\},\, \PP_{X',X}\{\tau > t\})].
    \end{equation}
    Since $\PP_{x,x'}\{\tau < \infty\} = 1$ by \eqref{eq:omb} for all $x, x' \in S$,
    both $\PP_{x,x'}\{\tau > t\}$ and $\PP_{x',x}\{\tau > t\}$ converge to $0$
    as $t \to \infty$ for each $(x,x')$.
    By the dominated convergence theorem, the right-hand side of
    \eqref{eq:wombd} converges to $0$.
    Since $\phi, \psi \in \pP(S)$ were arbitrary, $(P_t)$ is asymptotically
    contractive.
\end{proof}

\subsection{Proof of 
    \texorpdfstring{\cref{t:orgs}}{Theorem \ref{t:orgs}} and
    \texorpdfstring{\cref{t:orgsnew}}{Theorem \ref{t:orgsnew}}
}

First we prove the discrete time result in \cref{t:orgs}.

\begin{proof}[Proof of \cref{t:orgs}]
    Let $P$ be any stochastic kernel on $S$.
    If $P$ is globally stable on $\pP(S)$, then $P$ is bounded in probability, 
    since the sequence of distributions $t \mapsto \delta_x P^t = P^t(x, \cdot)$
    is convergent and hence tight. (By Prokhorov's theorem, every weakly convergent sequence of
    probability measures on a Polish space is tight.)

    Now suppose that $P$ is increasing, order reversing, and bounded in
    probability. By Lemma~A.5 from
    \cite{kamihigashi2014stochastic}, there exists
    an $(S \times S)$-valued Markov process $((X_t), (X_t'))_{t \in \ZZ_+}$ on 
    a probability space $(\Omega, \fF, \PP)$ such that 
    $(X_t)_{t \in \ZZ_+}$ and $(X_t')_{t \in \ZZ_+}$ are independent, both of
    these processes are $P$-Markov, and
    the stopping time $\tau \coloneq \inf \, \{t \geq 0 : X_t \leq X_t'\}$
    obeys $\PP_{x,x'} \{ \tau < \infty\} = 1$ for all $x, x' \in S$.
    Since $\preceq$ is closed, $\{(x,x') \in S \times S : x \leq x'\}$ is
    measurable, so $\tau$ is a stopping time.  The strong Markov property always holds
    in discrete time (see, e.g., \cite{meyn2009markov}).  As a result,
    the associated discrete time semigroup is weakly order mixing. 
    \cref{p:orwomnew} now implies that $(P_t)_{t \in \ZZ_+}$
    is asymptotically contractive on $\pP(S)$ with respect to the metric $\beta$.
    By definition, this means that the operator $P$ is asymptotically contractive on $\pP(S)$.
    Global stability now follows from \cref{c:act} (iii), given that
    boundedness in probability implies existence of a tight trajectory.
\end{proof}

Now we prove the continuous time result in \cref{t:orgsnew}.

\begin{proof}[Proof of \cref{t:orgsnew}]
    Since $(P_t)$ is increasing and weakly order mixing, \cref{p:orwomnew}
    implies that $(P_t)$ is asymptotically contractive on $(\pP(S), \beta)$.
    If $(P_t)$ is globally stable, then the trajectory $t \mapsto \phi P_t$
    is convergent for every $\phi \in \pP(S)$, and hence tight (since every
    weakly convergent sequence of probability measures on a Polish space is
    tight).  Conversely, if $(P_t)$ has at least one tight trajectory, then
    global stability follows from \cref{t:act}, since $(P_t)$ is
    order-preserving and asymptotically contractive.
\end{proof}

\subsection{Proof of \texorpdfstring{\cref{t:pdmpgs}}{Theorem \ref{t:pdmpgs}}}\label{ss:pdmpgsproof}

We prove this theorem using a sequence of lemmas.

\begin{lemma}\label{l:pdmpinc}
    If \cref{ass:flow} holds, then $(P_t)$ is increasing.
\end{lemma}

\begin{proof}
    Fix $x, x' \in S$ with $x \leq x'$.  On a common probability space,
    construct two PDMPs $(X_t)$ and $(X_t')$ starting from $x$ and $x'$
    respectively, using the same jump times $(T_n)$ and the same shocks
    $(\xi_n)$.  The embedded chains satisfy $Z_0 = x \leq x' = Z_0'$.
    By induction: if $Z_n \leq Z_n'$, then $\Phi(Z_n, E_{n+1}) \leq
    \Phi(Z_n', E_{n+1})$ by \cref{ass:flow}\,(i) and
    \begin{equation*}
        Z_{n+1}
        = F(\Phi(Z_n, E_{n+1}), \xi_{n+1})
        \leq F(\Phi(Z_n', E_{n+1}), \xi_{n+1})
        = Z_{n+1}'
    \end{equation*}
    by \cref{ass:flow}\,(ii).  Hence $Z_n \leq Z_n'$ for all $n$.
    For any $t \geq 0$, applying \cref{ass:flow}\,(i) gives
    \begin{equation*}
        X_t = \Phi(Z_{N_t}, t - T_{N_t}) \leq \Phi(Z_{N_t}', t - T_{N_t}) = X_t'.
    \end{equation*}
    Therefore, for any $h \in ibS$,
    $P_t h(x) = \EE_x h(X_t) \leq \EE_{x'} h(X_t') = P_t h(x')$.
\end{proof}

\begin{lemma}\label{p:pdmpbinp}
    If \cref{ass:flow,a:pdmpm} hold, then $(P_t)$ is bounded in probability.
\end{lemma}

The idea behind the proof is that the pathwise bounds in \cref{a:pdmpm} prevent
the post-jump states from escaping to infinity, while the inherently exponential
inter-arrival times prevent the deterministic flow from carrying the process too
far between jumps.

\begin{proof}
    Fix $x_0 \in S$ and $\varepsilon > 0$.  We need to find a compact set $D$
    such that $\PP_{x_0}\{X_t \in D\} \geq 1 - \varepsilon$ for all $t \geq 0$.
    By \cref{a:pdmpm},
    $f_1(\xi_n) \leq Z_n \leq f_2(\xi_n)$ for all $n \geq 1$, so each post-jump
    state lies in $[f_1(\xi_n), f_2(\xi_n)]$. Choose a compact $K \subset S$ with
    $\PP\{f_i(\xi) \in K\} \geq 1 - \varepsilon/6$ for $i = 1, 2$.  By
    \cref{a:pos}, $K \subset [a, b]$ for some $a, b \in S$, and $C \coloneq [a,
    b]$ is compact.  Whenever both $f_1(\xi_n)$ and $f_2(\xi_n)$ fall in $K$, the
    post-jump state $Z_n$ lies in $K \subset C$.

    Next, we show that $Z_{N_t}$, the state at the most recent jump before time
    $t$, lies in $C$ with high probability.  Formally, $\xi_n$ is independent of $(E_1, \ldots,
    E_{n+1})$ and $\{N_t = n\}$ is determined by $(E_1, \ldots, E_{n+1})$, so
    \begin{equation*}
        \PP_{x_0}\{Z_{N_t} \notin C, \, N_t \geq 1\}
        = \sum_{n=1}^{\infty} \PP\{N_t = n, \, Z_n \notin C\}
        \leq \frac{\varepsilon}{3}\sum_{n=1}^{\infty} \PP\{N_t = n\}
        = \frac{\varepsilon}{3}.
    \end{equation*}
    Between
    consecutive jumps, the process follows the flow $\Phi$, so $X_t =
    \Phi(Z_{N_t}, R_t)$ where $R_t \coloneq t - T_{N_t}$ is the time elapsed
    since the last jump.
    The event $\{R_t > M\}$ requires no Poisson arrivals in the interval
    $(t-M, t]$, so $\PP\{R_t > M\} \leq e^{-\lambda M}$.  Choose $M > 0$ with $e^{-\lambda M}
    < \varepsilon/3$.  Before the first jump ($N_t = 0$), the process
    follows its initial trajectory: $X_t = \Phi(x_0, t)$.  The probability of no
    jump by time $t$ is $\PP\{N_t = 0\} = e^{-\lambda t}$, which is small for
    large $t$.  Choose $T > 0$ with $e^{-\lambda T} < \varepsilon/3$.

    Since $\Phi$ is continuous,
    \begin{equation*}
        D \coloneq \Phi(C \times [0,M]) \cup \Phi(\{x_0\} \times [0,T])
    \end{equation*}
    is compact.  (The first component captures the flow starting from any
    post-jump state in $C$ and running for at most $M$ time units.  The second
    captures the initial trajectory from $x_0$ up to time $T$, covering the
    possibility that no jump has yet occurred.)
    Consider first the case $t \leq T$: if $N_t = 0$, then $X_t = \Phi(x_0, t) \in D$.
    For any such $t$: on the event $\{N_t \geq 1, \, Z_{N_t} \in C, \, R_t \leq M\}$,
    we have $X_t = \Phi(Z_{N_t}, R_t) \in D$.
    The three complementary events each have small probability:
    \begin{equation*}
        \PP_{x_0}\{X_t \notin D\}
        \leq \PP\{N_t = 0\} + \PP\{Z_{N_t} \notin C, \, N_t \geq 1\} + \PP\{R_t > M\}
        \leq e^{-\lambda t} + \frac{\varepsilon}{3} + e^{-\lambda M}.
    \end{equation*}
    For $t > T$ each term is at most $\varepsilon/3$, giving a total bound of
    $\varepsilon$.  For $t \leq T$ the first term covers the event $\{N_t = 0\}$
    (on which $X_t \in D$ by construction), so $\PP_{x_0}\{X_t \notin D\} \leq
    \varepsilon/3 + \varepsilon/3 < \varepsilon$.
\end{proof}

\begin{lemma}\label{l:pdmpwom}
    If \cref{ass:flow,a:pdmpm} hold, then $(P_t)$ is weakly order mixing.
\end{lemma}

\begin{proof}
    Construct two PDMPs $(X_t)$ and $(X_t')$ on a common probability space
    using the \emph{same} Poisson jump times $(T_n)$ but independent
    shock sequences $(\xi_n)$ and $(\xi_n')$, each {\sc iid} with common
    distribution $\mu$.  Each marginal process is $(P_t)$-Markov, and the joint
    process $((X_t), (X_t'))$ is itself a PDMP on $S \times S$ (with the same
    jump rate $\lambda$, joint flow $(\Phi(x,t), \Phi(x',t))$, and independent
    shocks), hence strong Markov.

    Fix $x, x' \in S$ and write $A_n = \{Z_n \leq Z_n'\}$ for the event that
    the two embedded chains become ordered at the $n$-th shared jump.
    On $A_n$, the semi-flow identity gives $X_{T_n} = \Phi(Z_n, 0) = Z_n \leq
    Z_n' = \Phi(Z_n', 0) = X_{T_n}'$.  Hence $\tau
    \leq T_n$, where $\tau \coloneq \inf\{t \geq 0 : X_t \leq X_t'\}$, and it
    suffices to show that $A_n$ occurs for some $n$ with probability one.

    Let $f_1$ and $f_2$ be as in \cref{a:pdmpm}
    and set $\delta \coloneq \PP\{f_2(\xi) \leq
    f_1(\xi')\} > 0$.  On the event $\{f_2(\xi_n) \leq
    f_1(\xi_n')\}$, we have $Z_n = F(\Phi(Z_{n-1}, E_n), \xi_n)
    \leq f_2(\xi_n) \leq f_1(\xi_n') \leq
    F(\Phi(Z_{n-1}', E_n), \xi_n') = Z_n'$, so $A_n$ occurs regardless of
    the pre-jump states.  Since $(\xi_n, \xi_n')$ is independent of all
    other randomness, $\PP\{A_n^c \mid \mathcal{F}_{n-1}\} \leq 1 - \delta$
    for every $n$.  Therefore
    \begin{equation*}
        \PP_{x,x'}\{\tau > T_k\}
        \leq \PP\{A_1^c \cap \cdots \cap A_k^c\}
        \leq (1 - \delta)^k.
    \end{equation*}
    Since $T_k \to \infty$ a.s., $\PP_{x,x'}\{\tau < \infty\} = 1$.
\end{proof}

\begin{proof}[Proof of \cref{t:pdmpgs}]
    By \cref{l:pdmpinc}, $(P_t)$ is increasing.
    By \cref{p:pdmpbinp}, $(P_t)$ is bounded in probability, which implies the
    existence of a tight trajectory.
    By \cref{l:pdmpwom}, $(P_t)$ is weakly order mixing.
    Global stability now follows from \cref{t:orgsnew}.
\end{proof}

\bibliographystyle{plainnat}

\end{document}